%% file: arxiv.tex
\documentclass[10pt,oneside]{article}
\usepackage[utf8]{inputenc}
\usepackage[T1]{fontenc}
\usepackage{lmodern,amsmath,amssymb,amsthm,bm}
\usepackage{booktabs,graphicx,microtype,parskip,needspace,tikz}
\usepackage[letterpaper,margin=1.1in]{geometry}
\usepackage[round]{natbib}
\usepackage[colorlinks=true,citecolor=blue,urlcolor=blue]{hyperref}
\hypersetup{pdftitle={On maximum likelihood estimation in the beta-model for random hypergraphs},pdfauthor={Milan Vojnovic and Yiliu Wang}}
\graphicspath{{figures/}}

\numberwithin{equation}{section}
\newtheorem{thm}{Theorem}[section]
\newtheorem{lm}[thm]{Lemma}
\newtheorem{cor}[thm]{Corollary}
\newtheorem{pro}[thm]{Proposition}
\theoremstyle{definition}
\newtheorem{definition}[thm]{Definition}
\newtheorem{ex}[thm]{Example}
\theoremstyle{remark}
\newtheorem{remark}[thm]{Remark}

\newcommand{\reals}{\mathbb{R}}
\newcommand{\E}{\mathbb{E}}
\renewcommand{\P}{\mathbb{P}}
\newcommand{\ind}{\mathbf{1}}

\newcommand{\rank}{\operatorname{rank}}
\newcommand{\tr}{\operatorname{tr}}
\newcommand{\conv}{\operatorname{conv}}
\newcommand{\ri}{\operatorname{ri}}
\newcommand{\Int}{\operatorname{int}}
\newcommand{\KL}{\operatorname{KL}}

\newcommand{\norm}[1]{\lVert #1\rVert}
\begin{document}
\title{On maximum likelihood estimation in the $\beta$-model for random hypergraphs}
\author{Milan Vojnovi\'c$^{1}$ and Yiliu Wang$^{2,3}$\thanks{Email correspondence: m.vojnovic@lse.ac.uk; yiliuw@uw.edu. Work carried out in part while Yiliu Wang was at the Department of Statistics, London School of Economics and Political Science.}\\[0.5ex]
{\small $^{1}$Department of Statistics, London School of Economics and Political Science}\\
{\small London, WC2A 2AE, United Kingdom}\\[0.5ex]
{\small $^{2}$Allen Institute, University of Washington, Seattle, Washington 98109}\\
{\small $^{3}$Department of Neurobiology \& Biophysics, University of Washington}\\
{\small Seattle, Washington 98195, United States}}
\date{}
\maketitle
\begin{abstract}
We study maximum likelihood estimation in the $\beta$-model for random hypergraphs, in which a group of vertices succeeds with a probability that depends on the sum of its members' parameters. The data are independent binary outcomes of experiments on groups of vertices, tested according to an arbitrary, possibly incomplete, design. Suppose that every tested group has $k$ vertices and success probability in $[\epsilon,1-\epsilon]$. For complete designs in which every group is tested at least $r$ times, a unique maximum likelihood estimator (MLE) exists with high probability if $r\geq C_k\epsilon^{-1}\log n/n^{k-1}$, and for $m$ experiments on uniformly random groups, it suffices that $m\geq C_k\epsilon^{-1}n\log n$, where $C_k$ depends only on $k$. Both conditions are linear in $\epsilon^{-1}$ and, for fixed $k$, cannot be improved in general beyond the constant; they follow from finite-sample bounds for arbitrary designs. We bound the mean squared error of the MLE over a parameter box by the smallest eigenvalue of the correlation matrix of the design, and show that the rate is minimax optimal for balanced complete designs. Among $k$-uniform designs with a given number of experiments, designs in which every pair of vertices takes part together in the same number of experiments maximize this eigenvalue, and uniform random designs nearly do so with high probability. Finally, the design matrix of a uniform random design has full column rank with high probability if $m\geq\max\{n,c(n/k)\log n\}$ for fixed $c>2$ and $k=o(n/\log n)$, or $m\geq n$ when $k\geq(1+\xi)\log n$. For fixed $k$, the threshold $(n/k)\log n$ is sharp, and above it rank deficiency is asymptotically as likely as some vertex taking part in no experiment.
\end{abstract}
\section{Introduction}
\input{intro}
\section{Problem formulation and preliminary results}\label{sec:model}
\input{model}
\section{MLE existence and uniqueness}\label{sec:existence}
\input{existence}
\section{Estimation error}\label{sec:accuracy}
\input{accuracy}
\section{Random designs}\label{sec:random}
\input{random}
\Needspace{8\baselineskip}
\section{Conclusion}
\input{conc}
\newpage

\bibliographystyle{abbrvnat}
\bibliography{references}
\clearpage
\begin{center}\textbf{\large Supplementary Material}\end{center}
\appendix
\input{appendix}
\section{Use of AI tools}
In 2026, the authors used Claude Opus 5.5 (Anthropic) and GPT-6 (OpenAI) to assist with revising the text, searching the literature, and checking proofs. Claude also assisted with developing the separating-vector approach of Section~\ref{sec:separation}, its application to random designs, and Proposition~\ref{pro:lambda-max}, and with writing the code for the numerical experiments. GPT-6 also assisted with completing the rank results of Section~\ref{sec:random-rank}, including the proof of Theorem~\ref{thm:rank-growing} based on \cite{FKS22}. The authors verified all results and take full responsibility for the content of the paper.

\end{document}

%% file: intro.tex
Many statistical problems ask for the strengths of individuals when only the outcomes of groups are observed. On a collaboration platform, the success or failure of a project reflects the skills of its team members \citep{KR18,SVY21,Mehta}; similar data arise in team sports, crowdsourcing, multiplayer games, and in group interactions studied in network science, neuroscience, and ecology \citep{diffusion,giusti,ecology}. A natural model for such data is the $\beta$-model for random hypergraphs, in which each tested group succeeds with a probability that depends on the sum of the strengths of its members. This is a logistic regression model whose design matrix is a sparse $0$--$1$ incidence matrix, with one row per experiment and one column per individual. This paper studies two basic questions about maximum likelihood estimation in this model: when the maximum likelihood estimator (MLE) exists and is unique, and how accurate it is. Both questions are governed by the design, that is, by which groups are tested and how many times.

The design is rarely under the control of the analyst, and it is usually sparse. The number of groups of size $k$ among $n$ individuals is $\binom nk$, so testing every group is infeasible even for moderate $n$ and $k$, and which groups are tested is often determined by the platform, by fixtures or matchmaking, or by a budget. Outcomes can also be sparse, when success or failure is rare. Existing results for the $\beta$-model mostly concern complete designs, in which every group of a given size is tested. Results that allow edge probabilities to approach zero or one either let them do so only slowly \citep{chatterjee2011,Yan2013} or rely on regularization or additional structure on the parameters \citep{sparsebeta,Shao2021L2}. We give finite-sample conditions for existence and accuracy that hold for arbitrary designs, evaluate them for complete designs and for designs in which the tested groups are sampled at random, and show that their dependence on the number of individuals, the probability margin, and replication cannot be improved in general.

\paragraph{The model}
The model is defined for $n\geq2$ vertices and a parameter vector $\bm\beta\in\reals^n$ such that a nonempty set $S\subseteq V:=\{1,\ldots,n\}$ is an edge with probability
\[
p_S(\bm\beta):=\sigma\left(\sum_{v\in S}\beta_v\right),
\]
where $\sigma(x):=1/(1+\exp(-x))$ is the logistic function and $\bm\beta$ is the unknown parameter. We interpret $\beta_v$ as the strength of vertex $v$. These edge probabilities were introduced by \cite{stasi}, who required each edge to contain at least two vertices; we also allow edges consisting of a single vertex, which represent individual outcomes. The model accommodates $k$-uniform hypergraphs, in which each edge has $k$ vertices. In particular, the case $k=2$ is the $\beta$-model for random graphs studied by \cite{chatterjee2011}, where $\{u,v\}$ is an edge with probability $\sigma(\beta_u+\beta_v)$. This is an undirected version of the $p_1$ model of \cite{Holland1981} and a generalization of the random graph model of \cite{ER59}.

\paragraph{The data}
An \emph{experiment} tests a group $S\subseteq V$ of vertices and has a binary outcome: it succeeds with probability $p_S(\bm\beta)$, independently of all other experiments. In the application above, an experiment is a project, the group is its team, and the outcome is the success or failure of the project. The \emph{design} specifies which groups are tested and how many times. We call the distinct tested groups \emph{candidate edges}, and the number of experiments on a candidate edge its \emph{replication number}. The candidate edges form a hypergraph on $V$ that describes the design, and the successful experiments form the observed random hypergraph, in which a candidate edge can appear several times. With one experiment on every $k$-subset of $V$, the observed hypergraph is a $k$-uniform random hypergraph of the kind considered by \cite{stasi}. A design is \emph{complete} if its candidate edges are all the $k$-subsets of $V$, for some $k$, and \emph{incomplete} otherwise; it is \emph{balanced} if all replication numbers are equal. Figure~\ref{fig:hypergraph} shows an incomplete design: panel~(a) shows its candidate edges and their replication numbers, and panel~(b) lists the resulting experiments and their outcomes.

The outcomes of the experiments are always random, while the design may be fixed or random. A fixed design is chosen before the outcomes are observed, and probabilities are then over the outcomes alone; this is the setting of Sections~\ref{sec:existence} and~\ref{sec:accuracy}. In a random design, studied in Section~\ref{sec:random}, the groups $S_1,\ldots,S_m$ are $k$-subsets of $V$ sampled independently and uniformly at random, and, conditional on the sampled design, the outcomes are independent, with experiment $j$ succeeding with probability $p_{S_j}(\bm\beta)$. There are then two sources of randomness, the draw of the groups and the outcomes given the design, and our results hold with high probability over both.

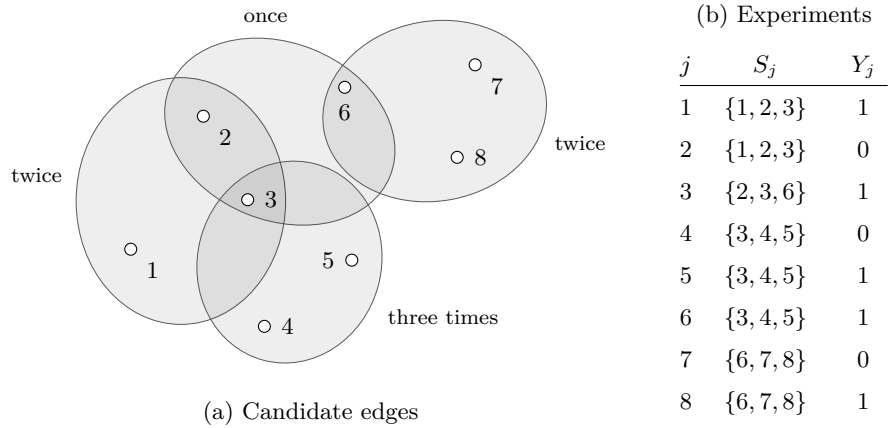
\begin{figure}[t]
\centering
\input{fig-design-outcomes}
\caption{An incomplete $3$-uniform design. (a)~The candidate edges are the four groups shown, and each is tested the indicated number of times. (b)~The resulting $m=8$ experiments, each consisting of a tested group $S_j$ and its outcome $Y_j$, with $Y_j=1$ for success.}
\label{fig:hypergraph}
\end{figure}

\paragraph{The role of the design}
A unique MLE exists if and only if the observed degrees, the numbers of successful experiments involving each vertex, lie in the interior of a polytope determined by the design. For example, the MLE does not exist if some vertex succeeds in all or none of its experiments, as the likelihood then increases strictly as the parameter of that vertex tends to $+\infty$ or $-\infty$, approaching a supremum that is not attained. The estimation accuracy is controlled by the correlation matrix of the design, whose entries count the experiments involving each pair of vertices; its smallest eigenvalue measures how well the tested groups distinguish the individual vertex parameters.

\subsection{Main results}

Throughout, $m$ denotes the total number of experiments, and $r$ denotes the least replication number of a complete design, so that a balanced complete $k$-uniform design has $m=r\binom nk$. Edge probabilities are assumed to lie in $[\epsilon,1-\epsilon]$ for some $\epsilon\in(0,1/2]$. A subscript on a constant indicates the parameters on which it may depend, and our order comparisons are for fixed $k$.
\begin{itemize}
\item \emph{MLE existence.} For complete $k$-uniform designs, a unique MLE exists with probability at least $1-1/n$ if $r\geq C_k\epsilon^{-1}\log n/n^{k-1}$ (Theorem~\ref{thm:hyper-sharp}). For uniform random designs, the same holds if $m\geq C_k\epsilon^{-1}n\log n$ (Theorem~\ref{thm:suffrnd}), which is the same condition as for balanced complete designs. Both conditions are necessary up to the value of the constant, by~\eqref{equ:complete-lower} and~\eqref{equ:rndlb}. These results follow from a bound for arbitrary designs (Proposition~\ref{pro:master}). We also give conditions through degree polytopes, for complete graph designs with unequal replication (Theorem~\ref{thm:mlesuff}), and through interior representations of the degree sequence, for all $2\leq k<n$ (Proposition~\ref{pro:projection-complete}).
\item \emph{Estimation error.} For any design whose groups have at most $k$ vertices, and the MLE $\hat{\bm\beta}_b$ over a box $[-b,b]^n$ containing the true parameter, the mean squared error per parameter is at most $1/(4\sigma'(kb)^2\lambda_1(\bm M))$, where $\lambda_1(\bm M)$ is the smallest eigenvalue of the correlation matrix (Theorem~\ref{lem:mse1}). For balanced complete designs this is of order $n/m$, which is minimax optimal (Proposition~\ref{pro:minimax}), and the same order holds with high probability for uniform random designs.
\item \emph{Rank of random designs.} For uniform random designs, the design matrix has full column rank with high probability if $m\geq\max\{n,c(n/k)\log n\}$, for $2\leq k=o(n/\log n)$ and any fixed $c>2$, and if $m\geq n$ when $k\geq(1+\xi)\log n$ (Theorem~\ref{thm:rank-growing}). For fixed $k$, the threshold is sharp at $m=(n/k)\log n$, and above it the probability of rank deficiency is asymptotically the probability that some vertex is in no experiment (Theorem~\ref{thm:rank-fixed}).
\end{itemize}

Two of these results deserve emphasis. First, to our knowledge, Theorems~\ref{thm:hyper-sharp} and~\ref{thm:suffrnd} give the first conditions on the parameters and the design under which a unique MLE exists with high probability in the $\beta$-model, for graphs or hypergraphs, that are linear in $\epsilon^{-1}$. For fixed $k$, they are matched, up to the value of the constant, by necessary conditions for a guarantee that holds for all edge probabilities in $[\epsilon,1-\epsilon]$. Such conditions differ from criteria that are checked on the observed degrees \citep{rinaldo2013,KarwaSlavkovic2016}. Earlier finite-sample conditions of this kind rely on concentration and, like our Corollary~\ref{thm:eps}, require a number of experiments quadratic in $\epsilon^{-1}$ \citep{rinaldo2013}, while asymptotic results take the margin $\epsilon$ fixed or let it tend to zero only slowly \citep{chatterjee2011,Yan2013,NandyBhattacharya2024}. Second, the same sufficient condition for MLE existence holds for uniform random designs as for balanced complete designs with the same number of experiments, and the estimation error is of the same order for both. Moreover, among all $k$-uniform designs with $m$ experiments, $\lambda_1(\bm M)$ is maximized by the designs in which every pair of vertices takes part together in the same number of experiments (Proposition~\ref{pro:lambda-max}), and uniform random designs come within a factor $1-\delta$ of the maximum with high probability once $m$ is of order at least $n\log n$ (Proposition~\ref{lem:lambdaconc}). In this sense, testing every group is unnecessary.

\subsection{Related work}

\paragraph{The $\beta$-model}
For random graphs with complete designs, \cite{chatterjee2011} showed uniform consistency of the MLE, and \cite{Yan2013} established asymptotic normality. \cite{rinaldo2013} characterized MLE existence through the polytope of graph degree sequences, allowing unequal replication numbers, and gave finite-sample probability bounds for a common replication number. \cite{KarwaSlavkovic2016} gave easily checked conditions for the existence of the MLE, and \cite{Hillar2013} extended the $\beta$-model to weighted graphs and established consistency of the MLE. \cite{Yan2016} studied moment estimation for a class of models containing the $\beta$-model, \cite{Mukherjee2016} identified detection thresholds on sparse graphs, and sparse networks were also studied through a sparse variant of the model \citep{sparsebeta} and through regularized estimators \citep{Shao2021L2}. The hypergraph $\beta$-model was introduced by \cite{stasi}, who provided iterative proportional scaling and fixed-point algorithms for the MLE and observed in simulations a transition from non-existence to existence as the edge density increases, without giving conditions for existence or bounds on the error. \cite{NandyBhattacharya2024} studied the layered hypergraph $\beta$-model, which allows edges of several sizes, and established existence, estimation rates, minimax optimality, asymptotic normality, and goodness-of-fit tests for complete designs with bounded parameters. In comparison, our existence bounds give explicit finite-sample probability estimates, quantify replication, and cover incomplete designs and edge probabilities approaching zero or one, and our estimation bounds express the effect of the design through its correlation matrix. Other statistical models for hypergraphs include models for folksonomies \citep{randhyper}, protein interactions \citep{murgas}, configuration and generative models \citep{bala,10.1093/comnet/cnaa018,marios,Yuan2022}, community detection \citep{ZhenWang2023,Ruggeri2023}, and latent space models \citep{YuZhu2025}; see \cite{Battiston2020} and \cite{Bick2023} for surveys.

\paragraph{Logistic regression and paired comparisons}
The separation criterion for the existence of the MLE in logistic regression is due to \cite{AA84}, and the convex-support formulation follows from the theory of exponential families \citep{bn78}. \cite{Candes-Sur-18} studied the existence threshold for logistic regression with Gaussian covariates, and \cite{Sur2019,Salehi2019} studied estimation in related high-dimensional regimes; the covariates here are instead sparse incidence vectors, which leads to different conditions. The closest analog is the Bradley--Terry model of paired comparisons, whose design vectors have the form $\bm e_u-\bm e_v$. There, the MLE exists if and only if the directed graph of outcomes is strongly connected \citep{Ford1957}, and estimation bounds involve the spectral gap of the Laplacian of the comparison graph \citep{HOX14,VY2016,SBBPRW16,HanLuoXu2026}. For the graph $\beta$-model, the design vectors are $\bm e_u+\bm e_v$, the corresponding matrix is the signless Laplacian, and its smallest eigenvalue is related to graph non-bipartiteness \citep{DR94,FF12}; for hypergraphs, the correlation matrix is not a Laplacian.

\paragraph{Ranks of random matrices}
Ranks of sparse random matrices have been studied by \cite{Costello2008,Costello2010,CFP19}. For fixed $k\geq3$, \cite{CFP19} established the threshold $(n/k)\log n$ for the maximum possible rank over the binary field. Identifiability requires instead full column rank over the real numbers, which is a different property: for even $k$, the all-ones vector is in the kernel over the binary field but not over the real numbers. Our proof for fixed $k$ works over a finite field whose characteristic exceeds $k$, which suffices because the minors of a $0$--$1$ matrix are integers, and a minor that vanishes over the real numbers also vanishes modulo the characteristic. For growing $k$, we use the square-matrix result of \cite{FKS22} together with a reduction based on sums of disjoint blocks of experiments.

\paragraph{Earlier versions of the results}
Some of the results in this paper appeared, in a different form, in the PhD thesis of the second author \citep{Wang2022Thesis}. Proposition~\ref{pro:projection-complete}, the separating-vector bounds of Section~\ref{sec:separation}, Propositions~\ref{pro:minimax} and~\ref{pro:lambda-max}, and Theorems~\ref{thm:rank-growing} and~\ref{thm:rank-fixed} are new; Theorem~\ref{thm:rank-growing} proves Conjecture~4.2.3 of \cite{Wang2022Thesis}. Corollary~\ref{cor:hyper}, Theorem~\ref{lem:mse1}, and Theorem~\ref{thm:suffrnd} sharpen results of the thesis, and the necessary conditions~\eqref{equ:complete-lower} and~\eqref{equ:rndlb} extend those of the thesis for graphs to $k$-uniform designs.

\subsection{Organization of the paper}

Section~\ref{sec:model} provides the problem formulation and preliminary results. Sections~\ref{sec:existence} and~\ref{sec:accuracy} concern fixed designs. Section~\ref{sec:existence} gives conditions for MLE existence and uniqueness, through degree polytopes, interior representations, and separating vectors, and a matching necessary condition. Section~\ref{sec:accuracy} bounds the estimation error in terms of the smallest eigenvalue of the correlation matrix and relates this eigenvalue to the structure of the design. Section~\ref{sec:random} turns to random designs and studies the rank of the design matrix, MLE existence, and estimation error. The proofs are in the Appendix.

%% file: fig-design-outcomes.tex
\begin{tikzpicture}[x=0.085mm,y=-0.085mm,
  edge/.style={draw=black!65,line width=0.35pt,fill=black,fill opacity=0.07},
  vtx/.style={circle,draw=black,fill=white,line width=0.4pt,inner sep=0pt,minimum size=4.5pt},
  rep/.style={font=\footnotesize,inner sep=1pt}]
\path[edge] (168,290) ellipse [x radius=163,y radius=192];
\path[edge] (322,182) [rotate around={-22:(322,182)}] ellipse [x radius=185,y radius=138];
\path[edge] (337,385) [rotate around={-15:(337,385)}] ellipse [x radius=143,y radius=158];
\path[edge] (563,150) [rotate around={10:(563,150)}] ellipse [x radius=175,y radius=140];
\foreach \v/\x/\y/\a in {1/90/365/-35,2/203/158/-35,3/272/288/0,4/298/485/0,5/434/382/180,6/423/113/-90,7/626/78/-35,8/598/222/0}
  \node[vtx,label={[label distance=0.5pt]\a:\small$\v$}] at (\x,\y) {};
\node[rep,anchor=east] at (-12,250) {twice};
\node[rep] at (300,5) {once};
\node[rep,anchor=west] at (485,470) {three times};
\node[rep,anchor=west] at (745,200) {twice};
\node[font=\small] at (370,620) {(a) Candidate edges};
\node[anchor=west,font=\small] at (930,300) {%
\renewcommand{\arraystretch}{1.05}%
\begin{tabular}{@{}ccc@{}}
\multicolumn{3}{c}{(b) Experiments}\\[3pt]
$j$ & $S_j$ & $Y_j$\\ \hline
1 & $\{1,2,3\}$ & 1\\
2 & $\{1,2,3\}$ & 0\\
3 & $\{2,3,6\}$ & 1\\
4 & $\{3,4,5\}$ & 0\\
5 & $\{3,4,5\}$ & 1\\
6 & $\{3,4,5\}$ & 1\\
7 & $\{6,7,8\}$ & 0\\
8 & $\{6,7,8\}$ & 1
\end{tabular}};
\end{tikzpicture}

%% file: model.tex
\paragraph{Problem formulation}
Let $V=\{1,\ldots,n\}$, $n\geq2$, and let $S_1,\ldots,S_m$ be a collection of nonempty subsets of $V$. We consider independent Bernoulli outcomes $Y_1,\ldots,Y_m$ with
\[
\P[Y_j=1]=\sigma\left(\sum_{v\in S_j}\beta_v\right),
\qquad \sigma(x)=\frac{1}{1+e^{-x}},
\]
where $\bm\beta\in\reals^n$ is the unknown parameter. We call $(S_j,Y_j)$ an experiment on the group $S_j$, and the list $S_1,\ldots,S_m$ the design. The same group may be tested in several experiments. A design is $k$-uniform if every candidate edge contains $k$ vertices; the graph case is $k=2$.

Let $E$ be the set of distinct groups among $S_1,\ldots,S_m$, which we call candidate edges, and let $M_S\geq1$ denote the number of experiments on $S\in E$. We call $M_S$ the replication number of $S$. A design is complete if $E$ contains all $k$-subsets of $V$, and balanced if every candidate edge has the same replication number $r$. When the replication numbers can differ, we write $r=\min_{S\in E}M_S$ for the least replication number. Thus a balanced complete design has $m=r\binom nk$ experiments. Designs that omit some candidate edges are called incomplete. In Sections~\ref{sec:existence} and~\ref{sec:accuracy}, the design is fixed and probabilities are over the outcomes. In Section~\ref{sec:random}, the design is random, and probabilities are over both.

For $\bm x_j=\ind_{S_j}\in\{0,1\}^n$, let $\bm X=(\bm x_1,\ldots,\bm x_m)^\top$ be the design matrix. The $\beta$-model is then logistic regression with $\P[Y_j=1]=\sigma(\bm x_j^\top\bm\beta)$. We define the unnormalized correlation (Gram) matrix
\[
\bm M=\bm X^\top\bm X=\sum_{j=1}^m\bm x_j\bm x_j^\top.
\]
Its entry $M_{u,v}$ is the number of experiments involving both vertices $u$ and $v$, and $M_{u,u}$ is the number involving $u$. For graph designs, $M_{u,v}$ is the replication number of the candidate edge $\{u,v\}$. We use $\lambda_1(\bm M)$ for its smallest eigenvalue, and $\norm{\cdot}$ for the Euclidean norm.

For observed outcomes $\bm y\in\{0,1\}^m$, the log-likelihood is
\begin{align}
\ell(\bm\beta)
&=\sum_{j=1}^m\bigl[y_j\log\sigma(\bm x_j^\top\bm\beta)
 +(1-y_j)\log(1-\sigma(\bm x_j^\top\bm\beta))\bigr] \notag\\
&=\bm d(\bm y)^\top\bm\beta-\kappa(\bm\beta),\label{equ:loglik}
\end{align}
where $\bm d(\bm y)=\bm X^\top\bm y$ is the sufficient statistic and $\kappa(\bm\beta)=\sum_j\log(1+e^{\bm x_j^\top\bm\beta})$ is the log-partition function. The coordinate $d_v(\bm y)=\sum_{j:v\in S_j}y_j$ is the number of successful experiments involving vertex $v$, that is, the degree of $v$ in the observed hypergraph formed by the successful experiments. We call $\bm d(\bm y)$ the \emph{degree sequence} of the outcomes. An MLE over a convex set $\Theta$ is a maximizer of $\ell$ on $\Theta$. Our existence results concern $\Theta=\reals^n$. For estimation bounds we also consider the compact parameter box $\Theta_b=\{\bm\beta':\norm{\bm\beta'}_\infty\leq b\}$.

\paragraph{Convexity and overlap}
\begin{lm}\label{lem:sc}
The negative log-likelihood is convex, and it is strictly convex if and only if $\rank(\bm X)=n$, equivalently $\lambda_1(\bm M)>0$.
\end{lm}
Indeed, its Hessian is
\begin{equation}
\nabla^2(-\ell)(\bm\beta)=\sum_{j=1}^m\sigma'(\bm x_j^\top\bm\beta)\bm x_j\bm x_j^\top,
\qquad \sigma'(z)>0.\label{equ:hessian}
\end{equation}
It is positive definite precisely when the columns of $\bm X$ are independent. If $\bm X\bm\alpha=0$, the likelihood is constant along $\bm\beta+t\bm\alpha$. Full rank ensures uniqueness when a maximizer exists; existence over $\reals^n$ also depends on the outcomes.

\begin{definition}\label{def:overlap}
A vector $\bm\alpha\in\reals^n\setminus\{\bm0\}$ is a separating vector for $(\bm X,\bm y)$ if, for every $j$,
\[
\bm x_j^\top\bm\alpha\geq0\quad\hbox{if }y_j=1,
\qquad
\bm x_j^\top\bm\alpha\leq0\quad\hbox{if }y_j=0.
\]
The pair $(\bm X,\bm y)$ satisfies the overlapping condition if there is no separating vector for $(\bm X,\bm y)$.
\end{definition}

\begin{pro}\label{thm:mle-overlap}
There exists a unique finite MLE over $\reals^n$ if and only if $(\bm X,\bm y)$ satisfies the overlapping condition.
\end{pro}
This is the separation criterion of \cite{AA84}. In particular, overlap implies full column rank: every nonzero vector in the kernel of $\bm X$ would be a separating vector. The direct proof in Appendix~\ref{app:prelim} shows that the negative log-likelihood does not increase along a separating vector, so that it has no unique minimizer, and that under overlap it grows at least linearly in $\norm{\bm\beta}$, so that it attains its minimum.

\paragraph{Degree polytopes}
The convex support of $\bm d(\bm Y)$ is
\begin{equation}
C_{\bm X}=\conv\{\bm X^\top\bm z:\bm z\in\{0,1\}^m\}
=\bm X^\top[0,1]^m
=\sum_{S\in E}[\bm0,M_S\ind_S].\label{equ:support}
\end{equation}
Here $\bm X^\top[0,1]^m=\{\bm X^\top\bm z:\bm z\in[0,1]^m\}$ is the image of the unit cube $[0,1]^m$, $[\bm0,\bm a]=\{t\bm a:t\in[0,1]\}$ is the line segment from $\bm0$ to $\bm a$, and a sum of sets is the Minkowski sum $A+B=\{\bm a+\bm b:\bm a\in A,\ \bm b\in B\}$. The first equality holds because a linear map commutes with taking convex hulls. The second follows by grouping the experiments by candidate edge: the $M_S$ coordinates of $\bm z$ for experiments on $S$ contribute $(\sum_{j:S_j=S}z_j)\ind_S$, and this sum of coordinates ranges over $[0,M_S]$. Thus $C_{\bm X}$ is a sum of line segments, one for each candidate edge.
The exponential-family criterion states that a finite MLE exists if and only if $\bm d(\bm y)\in\ri C_{\bm X}$ \citep{bn78}. It is unique if in addition $\rank(\bm X)=n$, in which case relative interior is ordinary interior.

Let $\tilde y_S=\sum_{j:S_j=S}y_j$ be the number of successful experiments on $S\in E$, so that $\bm d(\bm y)=\sum_{S\in E}\tilde y_S\ind_S$. Define the normalized degree sequence and the corresponding polytope by
\begin{equation}
\tilde{\bm d}(\bm y)=\sum_{S\in E}\frac{\tilde y_S}{M_S}\ind_S,
\qquad Z(E)=\sum_{S\in E}[\bm0,\ind_S].\label{equ:normdegree}
\end{equation}
The polytope $Z(E)$ is obtained from $C_{\bm X}$ by rescaling the segment of each candidate edge $S$ to unit length, and it does not depend on the replication numbers. The normalized degree sequence lies in the relative interior of $Z(E)$ exactly when $\bm d(\bm y)$ lies in the relative interior of $C_{\bm X}$. Indeed, for any $\bm\alpha\in\reals^n$, the point $\bm d(\bm y)$ maximizes $\bm\alpha^\top\bm z$ over $\bm z\in C_{\bm X}$, and the point $\tilde{\bm d}(\bm y)$ maximizes $\bm\alpha^\top\bm z$ over $\bm z\in Z(E)$, in both cases if and only if $\tilde y_S=M_S$ for every $S$ with $\bm\alpha^\top\ind_S>0$ and $\tilde y_S=0$ for every $S$ with $\bm\alpha^\top\ind_S<0$. Hence $\bm d(\bm y)$ lies on a proper face of $C_{\bm X}$ if and only if $\tilde{\bm d}(\bm y)$ lies on a proper face of $Z(E)$. The replication numbers therefore change the shape of $C_{\bm X}$ but not this interior criterion. For complete graph designs, the normalized degree sequence and this criterion are due to \citet[Theorem~3.1]{rinaldo2013}, who prove it through the same representation of the two polytopes as sums of line segments; the argument applies unchanged to arbitrary designs.

For complete $k$-uniform designs, $Z(E)$ is the polytope $P_{n,k}$ of degree sequences of simple $k$-uniform hypergraphs. This follows by taking the convex hull of the incidence sums with coefficients in $\{0,1\}$. For graphs we write $P_n=P_{n,2}$. For incomplete designs the corresponding polytope is $Z(E)$, determined by the tested candidate edges.

\paragraph{Erd\H{o}s--Gallai inequalities}
A nonnegative integer sequence $d_1,\ldots,d_n$ is a simple graph degree sequence if and only if its sum is even and, for every nonempty $S\subseteq V$,
\begin{equation}
\sum_{v\in S}d_v\leq |S|(|S|-1)+\sum_{v\notin S}\min\{d_v,|S|\}.\label{equ:egcond}
\end{equation}
This is the Erd\H{o}s--Gallai characterization \citep{erdosgallai}. The polytope $P_n$ of graph degree sequences admits the linear inequality description \citep{PS89}
\begin{equation}
P_n=\bigl\{\bm d\in\reals^n:\ 0\leq d_v\leq n-1\hbox{ for all }v\in V,\ \ f(S,T,\bm d,n)\geq0\hbox{ for all }(S,T)\in\Omega\bigr\},\label{equ:eglinear}
\end{equation}
where
\begin{align}
f(S,T,\bm d,n)&=|S|(n-1-|T|)-\sum_{v\in S}d_v+\sum_{v\in T}d_v,\label{equ:fst}\\
\Omega&=\{(S,T):S,T\subseteq V,\ S\cap T=\emptyset,\ S\cup T\ne\emptyset\}.\label{equ:omega}
\end{align}
We use the full inequality family, which can include redundant inequalities. For $n\geq3$, membership in $\Int P_n$ is equivalent to strict satisfaction of these inequalities. This gives an explicit degree-based existence criterion for complete graph designs. The next section develops its finite-sample consequences and extends the degree-polytope approach to hypergraphs.

%% file: existence.tex
In this section, we provide sufficient conditions for MLE existence and uniqueness for designs that are fixed before the outcomes are observed. We first consider complete graph designs and then complete $k$-uniform hypergraph designs. Our results use the degree-polytope formulation introduced in Section~\ref{sec:model}; a further bound based on separating vectors quantifies the effect of sparse edge probabilities and applies to arbitrary designs.

\subsection{A degree-polytope condition for complete graph designs}\label{sec:existence-main}
We use two parameters: one describes the distance of the expected normalized degree sequence from the inequalities defining $P_n$, and the other describes the replication of experiments.
For $\bm d\in[0,n-1]^n$, let
\begin{equation}
\mathcal E(\bm d)=\frac{1}{n-1}\min\left\{
\min_v d_v,\ \min_v(n-1-d_v),\
\min_{(S,T)\in\Omega}\frac{f(S,T,\bm d,n)}{|S\cup T|}\right\}.\label{equ:margin}
\end{equation}
Thus $(n-1)\mathcal E(\bm d)$ is a normalized inequality margin, and $\bm d\in\Int P_n$ if and only if $\mathcal E(\bm d)>0$. For complete graph designs, define
\[
H_u(\bm M)=\frac{n-1}{\sum_{v\ne u}1/M_{u,v}},\qquad
H(\bm M)=\min_u H_u(\bm M).
\]
This is the least harmonic mean of the incident replication numbers. The normalized degree sequence has mean $\E[\tilde d_u]=\sum_{v\ne u}p_{u,v}(\bm\beta)$.

\begin{thm}\label{thm:mlesuff}
For a complete graph design with $n\geq3$, for any $c>1/2$, a unique MLE exists with probability at least $1-2/n^{2c-1}$ provided that
\begin{equation}
\mathcal E(\E[\tilde{\bm d}])\geq
\sqrt{\frac{c}{H(\bm M)}\frac{\log n}{n-1}}.\label{equ:edcond}
\end{equation}
For a balanced complete design, $H(\bm M)=r$ and $m=r\binom n2$, so the condition is equivalent to
\[
m\geq\frac{c}{2}\,\frac{1}{\mathcal E(\E[\tilde{\bm d}])^2}\,n\log n.
\]
\end{thm}

Theorem~\ref{thm:mlesuff}, like Theorem~5.1 of \cite{rinaldo2013}, concerns complete graph designs. It extends their result, which assumes a common replication number, to unequal replication numbers, through the harmonic mean $H(\bm M)$; it does not extend to hypergraph designs, for the reason explained in Section~\ref{sec:complete-hyper}. The proof in Appendix~\ref{app:mlesuff} follows theirs and uses the Erd\H{o}s--Gallai degree-polytope inequalities. If any inequality is tight at $\tilde{\bm d}$, then $\norm{\tilde{\bm d}-\E\tilde{\bm d}}_\infty$ must be at least $(n-1)\mathcal E(\E\tilde{\bm d})$. Hoeffding's inequality bounds this deviation in terms of $H(\bm M)$.

\begin{cor}\label{thm:eps}
Assume $n\geq3$, a complete graph design with $r=\min_{u\ne v}M_{u,v}$, and $\epsilon\leq p_{u,v}(\bm\beta)\leq1-\epsilon$ for some $\epsilon\in(0,1/2]$. For any $c>1/2$, a unique MLE exists with probability at least $1-2/n^{2c-1}$ if
\[
r\geq16c\,\frac{1}{\epsilon^2}\,\frac{\log n}{n-1}.
\]
For balanced designs this is $m\geq8c\epsilon^{-2}n\log n$.
\end{cor}
The proof in Appendix~\ref{app:mlesuff} shows that $\mathcal E(\E\tilde{\bm d})\geq\epsilon/4$. Since $H(\bm M)\geq r$, the stated condition on $r$ then implies~\eqref{equ:edcond}, and Theorem~\ref{thm:mlesuff} applies. The bound on $\mathcal E(\E\tilde{\bm d})$ holds because every pair $\{u,v\}$ contributes at least $\epsilon$ to each inequality margin. The coordinate margins $\E \tilde d_v$ and $n-1-\E \tilde d_v$ are sums of $n-1$ terms $p_{u,v}$ and $1-p_{u,v}$, respectively. The mixed margin $f(S,T,\E\tilde{\bm d},n)$ is a sum of terms $1-p_{u,v}$ over pairs meeting $S$ but not $T$, and terms $p_{u,v}$ over pairs meeting $T$ but not $S$, with pairs inside $S$ or inside $T$ counted twice. Counting these pairs gives $f(S,T,\E\tilde{\bm d},n)\geq\epsilon((n-1)|S\cup T|-2|S||T|)$, which is at least $\epsilon(n-1)|S\cup T|/4$ for $n\geq3$.

\subsection{Interior representations for complete hypergraph designs}\label{sec:complete-hyper}
We consider complete $k$-uniform hypergraph experiments. The degree-polytope formulation continues to hold, with $P_n$ replaced by $P_{n,k}$. For $k\geq3$, however, no complete linear description of $P_{n,k}$ is known. \cite{MurthySrinivasan2002} give partial information on its facets. In particular, by their Corollary~3.5 and Example~3.3(ii), every $\bm d\in P_{n,k}$ satisfies, for all disjoint $S,T\subseteq V$,
\begin{equation}
\sum_{v\in S}d_v-(k-1)\sum_{v\in T}d_v\leq|S|\binom{n-1-|T|}{k-1},\label{equ:ms-ineq}
\end{equation}
and the right-hand side is the maximum of the left-hand side over $P_{n,k}$. For $k=2$, \eqref{equ:ms-ineq} is the Erd\H{o}s--Gallai inequality $f(S,T,\bm d,n)\geq0$ in~\eqref{equ:fst}, and \citet[Theorem~3.12]{MurthySrinivasan2002} show that, for $k\geq3$, many of the inequalities~\eqref{equ:ms-ineq} define facets of $P_{n,k}$. Since these inequalities do not describe $P_{n,k}$, they do not yield an analog of Theorem~\ref{thm:mlesuff}, and we use a different argument.

To obtain sufficient existence conditions, we use a description of the interior of the convex support. Since $C_{\bm X}=\bm X^\top[0,1]^m$ by~\eqref{equ:support}, and a linear map takes the relative interior of a convex set onto the relative interior of its image, $\bm d(\bm y)\in\ri C_{\bm X}$ if and only if
\[
\bm d(\bm y)=\bm X^\top\bm w\qquad\hbox{for some }\bm w\in(0,1)^m.
\]
We call such $\bm w$ an interior representation of $\bm d(\bm y)$. For complete graph designs, this description of the interior is given by \citet[Lemma~3.2]{rinaldo2013}. It expresses the observed degrees as the degrees of a hypergraph whose edge weights all lie strictly between $0$ and $1$, whereas the outcome vector $\bm y$ itself gives a representation with weights in $\{0,1\}$. By the criterion of Section~\ref{sec:model}, a finite MLE exists if and only if $\bm d(\bm y)\in\ri C_{\bm X}$, that is, if and only if an interior representation exists. By Lemma~\ref{lem:sc}, it is unique if, in addition, $\bm X$ has full column rank. For complete designs, this rank condition always holds.
\begin{lm}\label{lem:complete-rank}
For $2\leq k<n$, every complete $k$-uniform design has full column rank.
\end{lm}
The proof in Appendix~\ref{app:hyper} computes the correlation matrix of one experiment on every $k$-set and shows that it is positive definite. Hence, for a complete design, an interior representation of $\bm d(\bm y)$ implies that a unique MLE exists. Below, $\binom{n-1}{k-1}$ is the number of $k$-subsets of $V$ that contain a given vertex, and $\binom nk=\frac nk\binom{n-1}{k-1}$.

\begin{pro}\label{pro:projection-complete}
For a complete $k$-uniform design with $2\leq k<n$ and replication number at least $r$ for every candidate edge, suppose $\epsilon\leq p_S(\bm\beta)\leq1-\epsilon$ for all $k$-sets $S$, where $\epsilon\in(0,1/2]$. Then
\begin{equation}
\P[\text{no unique finite MLE}]
\leq2\binom nk\exp\left(-\frac{2r\epsilon^2}{k}\binom{n-1}{k-1}\right).\label{equ:projection-bound}
\end{equation}
In particular, for any $c>0$, success probability at least $1-n^{-c}$ is ensured by
\begin{equation}
r\geq\frac{1}{2}\,\frac{1}{\epsilon^2}\,\frac{k\bigl(\log\bigl(2\binom nk\bigr)+c\log n\bigr)}{\binom{n-1}{k-1}}.\label{equ:projection-condition}
\end{equation}
\end{pro}

The proof, given in Appendix~\ref{app:hyper}, constructs an interior representation explicitly. Let $p_j=p_{S_j}(\bm\beta)$ be the success probability of experiment $j$, and let $\bm p=(p_1,\ldots,p_m)$. Since $\bm X^\top\bm p=\E[\bm d(\bm Y)]$, the vector $\bm p$ represents the expected degrees, with weights in $[\epsilon,1-\epsilon]$. To represent the observed degrees instead, we add the correction
\[
\bm q=\bm X\bm M^{-1}\bm X^\top(\bm y-\bm p),
\]
which is the vector of least Euclidean norm satisfying $\bm X^\top\bm q=\bm d(\bm y)-\bm X^\top\bm p$. Then $\bm w=\bm p+\bm q$ satisfies $\bm X^\top\bm w=\bm d(\bm y)$. If $\norm{\bm q}_\infty<\epsilon$, every coefficient of $\bm w$ lies in $(0,1)$, so $\bm w$ is an interior representation of $\bm d(\bm y)$ and a unique MLE exists. Appendix~\ref{app:hyper} also verifies this directly by showing that the negative log-likelihood is coercive, that is, tends to infinity as $\norm{\bm\beta}\to\infty$.

It remains to bound the probability that $|q_j|\geq\epsilon$ for some $j$. Each $q_j$ is a linear combination of the independent centered outcomes $y_l-p_l$, and its squared coefficients sum to $h_j=\bm x_j^\top\bm M^{-1}\bm x_j$. Hoeffding's inequality therefore gives $\P[|q_j|\geq\epsilon]\leq2\exp(-2\epsilon^2/h_j)$. For a complete design with one experiment on every $k$-set, the values $h_j$ are equal by symmetry and sum to $n$, so each equals $n/\binom nk=k/\binom{n-1}{k-1}$. Replication number at least $r$ for every candidate edge reduces them to at most $k/(r\binom{n-1}{k-1})$. Experiments on the same candidate edge have the same value of $q_j$, and a union bound over the $\binom nk$ candidate edges gives~\eqref{equ:projection-bound}.

We now discuss the right-hand side of~\eqref{equ:projection-condition}, starting with graphs. For $k=2$, we have $\binom{n-1}{k-1}=\binom{n-1}{1}=n-1$ and $\log(2\binom n2)\leq2\log n$, so~\eqref{equ:projection-condition} is implied by
\[
r\geq(c+2)\,\frac{1}{\epsilon^2}\,\frac{\log n}{n-1}.
\]
This matches the following heuristic. In a balanced design, each vertex takes part in $r(n-1)$ experiments, and its expected degree is at distance at least $\epsilon r(n-1)$ from the bounds $0$ and $r(n-1)$. Uniformly over the vertices, the observed degrees deviate from their means by order $\sqrt{r(n-1)\log n}$. The condition requires these deviations to be small compared with the distance $\epsilon r(n-1)$. Corollary~\ref{thm:eps} gives the same form with a larger constant: $r\geq8(c+1)\epsilon^{-2}\log n/(n-1)$ for success probability at least $1-2n^{-c}$. Corollary~\ref{cor:graph-sharp} below improves the factor $\epsilon^{-2}$ to $\epsilon^{-1}$.

For fixed $k\geq2$, we have $\binom{n-1}{k-1}=n^{k-1}(1+o(1))/(k-1)!$ and $\log(2\binom nk)=k\log n(1+o(1))$. The right-hand side of~\eqref{equ:projection-condition} is therefore
\[
\frac{1}{2}\,\frac{1}{\epsilon^2}\,\frac{(k+c)\,k!\,\log n}{n^{k-1}}(1+o(1)).
\]
The heuristic is the same as for graphs, with each vertex taking part in $r\binom{n-1}{k-1}$ experiments. The right-hand side decreases with $n$ and falls below one once $\epsilon^2n^{k-1}/\log n$ is large enough. In that regime, a single experiment on every candidate edge suffices. Bounding the binomial coefficients explicitly gives the following condition for this case.
\Needspace{6\baselineskip}
\begin{cor}\label{cor:hyper}
For a complete $k$-uniform design with $k\geq2$ and $n>4k$, and parameter $\bm\beta$ such that $\epsilon\leq p_S(\bm\beta)\leq1-\epsilon$ for every $k$-set $S$, for some $\epsilon\in(0,1/2]$, a unique MLE exists with probability at least $1-n^{-c}$, for any $c>0$, provided that
\begin{equation}
\epsilon\geq\sqrt{2\bigl((c+1)\log n+\log4\bigr)}\left(\frac{k}{n}\right)^{(k-1)/2}.\label{equ:epscond}
\end{equation}
\end{cor}

Corollary~\ref{cor:hyper} follows from Proposition~\ref{pro:projection-complete} with $r=1$, as shown in Appendix~\ref{app:hyper-constants}. It holds for every replication number, but it does not benefit from replication. For fixed $k$, the right-hand side of~\eqref{equ:epscond} is of order $\sqrt{\log n/n^{k-1}}$, the same order as obtained by solving~\eqref{equ:projection-condition} for $\epsilon$ with $r=1$. Under $n\geq2k$, the following separating-vector argument improves the factor $\epsilon^{-2}$ in~\eqref{equ:projection-condition} to $\epsilon^{-1}$. Proposition~\ref{pro:projection-complete} also applies when $n/2<k<n$.

\subsection{Sharper conditions via separating vectors}\label{sec:separation}
The conditions of Sections~\ref{sec:existence-main} and~\ref{sec:complete-hyper} are obtained by concentration. They require the observed degrees to stay within a margin proportional to $\epsilon$ of their means, and they lead to conditions of the form $r\geq C_k\epsilon^{-2}\log n/n^{k-1}$ for complete designs. In this section, we derive conditions that are sharper in their dependence on the edge-probability margin $\epsilon$. For complete designs with $n\geq2k$, they require $r\geq C_k\epsilon^{-1}\log n/n^{k-1}$, with $\epsilon^{-1}$ in place of $\epsilon^{-2}$. The improvement matters when some edge probabilities are close to zero or one. It comes from computing the probabilities of the events on which the MLE fails to exist, rather than bounding deviations of the degrees from their means; we explain this after Theorem~\ref{thm:hyper-sharp}. We also give a necessary condition showing that, for fixed $k$, this dependence on $\epsilon$, $n$, and $r$ cannot be improved in general. The argument works directly with the separating vectors of Definition~\ref{def:overlap}, and its first step applies to arbitrary designs.

For disjoint $P,N\subseteq V$, not both empty, put $Z_0=V\setminus(P\cup N)$ and define
\begin{equation}
W(P,N)=\sum_{\substack{S\in E:S\subseteq P\cup Z_0\\S\cap P\ne\emptyset}}M_S
+\sum_{\substack{S\in E:S\subseteq N\cup Z_0\\S\cap N\ne\emptyset}}M_S.\label{equ:W}
\end{equation}
This counts the experiments on groups that meet one of the two sets and avoid the other. Figure~\ref{fig:separation} illustrates the three sets and the classification of the groups for a $3$-uniform design.
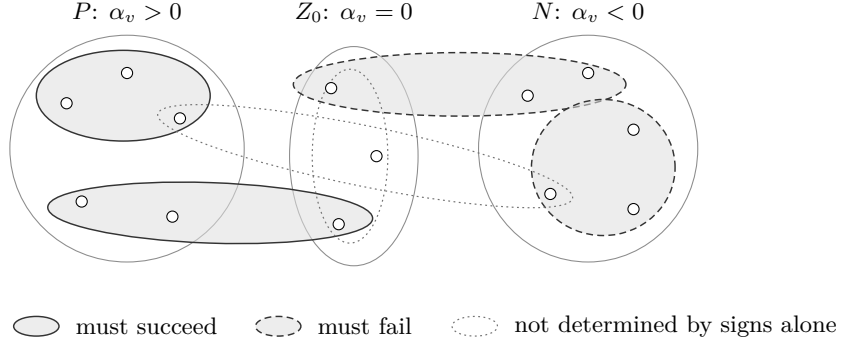
\begin{figure}[t]
\centering
\input{fig-separation}
\caption{The sign pattern of a separating vector $\bm\alpha$ for a $3$-uniform design, with groups drawn as ellipses. The vertices are split into $P$, $Z_0$, and $N$ according to the sign of $\alpha_v$. An experiment on a group that meets $P$ and avoids $N$ must succeed (solid outline), and an experiment on a group that meets $N$ and avoids $P$ must fail (dashed outline); these are the experiments counted by $W(P,N)$. For the other groups (dotted outline), the outcome is not determined by the signs alone: for a group that meets both $P$ and $N$, it depends on the values of $\bm\alpha$, and a group inside $Z_0$ imposes no constraint.}
\label{fig:separation}
\end{figure}

\begin{lm}\label{lm:certificate}
If a separating vector $\bm\alpha\ne\bm0$ exists, let $P=\{v:\alpha_v>0\}$ and $N=\{v:\alpha_v<0\}$. Every experiment in the first sum of~\eqref{equ:W} has outcome $1$, and every experiment in the second sum has outcome $0$.
\end{lm}
The proof is in Appendix~\ref{app:separation}. The lemma reduces a continuum of separating vectors to finitely many outcome events, one for each pair $(P,N)$.

\begin{pro}\label{pro:master}
For any fixed design with $\epsilon\leq p_S(\bm\beta)\leq1-\epsilon$ on every tested group,
\begin{equation}
\P[\text{no unique finite MLE}]
\leq\sum_{\substack{P,N\subseteq V:P\cap N=\emptyset\\P\cup N\ne\emptyset}}
(1-\epsilon)^{W(P,N)}.\label{equ:master}
\end{equation}
For a random design such that, conditional on the design, the outcomes are independent with success probabilities $p_{S_j}(\bm\beta)$, the same bound holds with each summand replaced by $\E[(1-\epsilon)^{W(P,N)}]$.
\end{pro}
The proof in Appendix~\ref{app:separation} combines Proposition~\ref{thm:mle-overlap} and Lemma~\ref{lm:certificate} with a union bound over sign patterns. It applies to complete and incomplete designs.

\Needspace{9\baselineskip}
\begin{thm}\label{thm:hyper-sharp}
Let $n\geq2k$, $k\geq2$, and consider complete $k$-uniform experiments with replication number at least $r$ for every candidate edge. Suppose $\epsilon\leq p_S(\bm\beta)\leq1-\epsilon$ for all $k$-sets $S$. For any $a>0$, a unique MLE exists with probability at least $1-n^{-a}$ if
\begin{equation}
r\geq\frac{1}{\epsilon}\,\frac{k\bigl((1+a)\log n+3\bigr)}{\binom{\lfloor n/2\rfloor-1}{k-1}}.\label{equ:hyper-sharp}
\end{equation}
\end{thm}
For fixed $k$, keeping only the leading term in $n$ of its right-hand side, condition~\eqref{equ:hyper-sharp} reads
\[
r\geq(1+a)\,\frac{1}{\epsilon}\,\frac{2^{k-1}k!\,\log n}{n^{k-1}}(1+o(1)).
\]
For graphs, the constant is $4(1+a)$, and Corollary~\ref{cor:graph-sharp} below reduces it to $2(1+a)$. With one experiment on every candidate edge, it suffices that $\epsilon$ is of order $\log n/n^{k-1}$, compared with $\sqrt{\log n/n^{k-1}}$ in Corollary~\ref{cor:hyper}.

The proof in Appendix~\ref{app:separation} lower bounds $W(P,N)$ by
\[
\frac{r(|P|+|N|)}{k}\binom{\lfloor n/2\rfloor-1}{k-1}.
\]
It then sums over at most $\binom ns2^s$ choices with $|P|+|N|=s$.

A factor $\epsilon^{-2}$ is what one expects from concentration inequalities, and it is less obvious what allows $\epsilon^{-1}$. Two facts are used. First, the failure event is contained in a finite union of prescribed-outcome events. By Proposition~\ref{thm:mle-overlap} and Lemma~\ref{lm:certificate}, a unique MLE fails to exist only if, for some sign pattern $(P,N)$, all $W(P,N)$ experiments counted in~\eqref{equ:W} have prescribed outcomes. Second, the probability of such an event is a product over independent experiments. Every experiment has each of its two outcomes with probability at least $\epsilon$, so each factor is at most $1-\epsilon\leq e^{-\epsilon}$, and the event has probability at most $e^{-\epsilon W(P,N)}$. No deviation from a mean has to be controlled, and each experiment with a prescribed outcome contributes a factor $e^{-\epsilon}$, compared with $e^{-2\epsilon^2}$ in Hoeffding's inequality.

The simplest case shows the difference. Suppose every edge probability equals $\epsilon$, and consider the separating vector $-\bm e_v$. It occurs exactly when all $w\geq r\binom{n-1}{k-1}$ experiments involving $v$ fail, that is, when the degree of $v$ is zero, a deviation of $\epsilon w$ below its mean. Hoeffding's inequality bounds the probability of this deviation by $\exp(-2\epsilon^2w)$, because it uses only that the outcomes lie in $[0,1]$, which allows a variance as large as $1/4$. The actual probability is $(1-\epsilon)^w\leq\exp(-\epsilon w)$. The degree of $v$ has standard deviation $\sqrt{\epsilon(1-\epsilon)w}$, so a deviation of $\epsilon w$ is only about $\sqrt{\epsilon w}$ standard deviations, and the exponent is linear in $\epsilon$.

The remaining step is the lower bound on $W(P,N)$, which grows linearly in $|P|+|N|$. This makes the union bound over the at most $\binom ns2^s$ sign patterns with support $s$ a convergent geometric series. The term $\log n$ in~\eqref{equ:hyper-sharp} is the cost of this union bound per vertex of the support.

For graphs, the exact count of constrained edges gives a smaller constant.
\begin{cor}\label{cor:graph-sharp}
For complete graph experiments with $n\geq40$, $r=\min_{u\ne v}M_{u,v}$, and $\epsilon\leq p_{u,v}(\bm\beta)\leq1-\epsilon$, for any $a>0$, a unique MLE exists with probability at least $1-n^{-a}$ if
\[
r\geq2\,\frac{1}{\epsilon}\,\frac{(1+a)\log n+3}{n-1}.
\]
For balanced experiments this is $m\geq\epsilon^{-1}n\bigl((1+a)\log n+3\bigr)$.
\end{cor}
The proof in Appendix~\ref{app:separation} follows that of Theorem~\ref{thm:hyper-sharp}, with the incidence bound on $W(P,N)$ replaced by the exact count $W(P,N)\geq r\bigl(\binom{|P|}2+\binom{|N|}2+s(n-s)\bigr)$, where $s=|P|+|N|$. The resulting union bound is controlled by a concavity argument in $s$.

\paragraph{A necessary condition}
We now show that the dependence of Theorem~\ref{thm:hyper-sharp} on $n$, $\epsilon$, and $r$ cannot be improved in general. To this end, we give a condition that any guarantee of the form of Theorem~\ref{thm:hyper-sharp}, holding for all parameters with edge probabilities in $[\epsilon,1-\epsilon]$, must satisfy. Consider a balanced complete $k$-uniform design with replication number $r$, and take $\beta_v=k^{-1}\log(\epsilon/(1-\epsilon))$ for every vertex, so that every edge has probability $\epsilon$. This parameter belongs to the class of Theorem~\ref{thm:hyper-sharp}. A fixed vertex $v$ takes part in $r\binom{n-1}{k-1}$ experiments. If all of them fail, then $-\bm e_v$ is a separating vector, and no unique finite MLE exists. This event has probability $(1-\epsilon)^{r\binom{n-1}{k-1}}$. Hence, if a unique MLE exists with probability at least $1-n^{-a}$ for this parameter, then $(1-\epsilon)^{r\binom{n-1}{k-1}}\leq n^{-a}$, which is equivalent to the necessary condition
\begin{equation}
r\geq a\,\frac{1}{\log(1/(1-\epsilon))}\,\frac{\log n}{\binom{n-1}{k-1}}.\label{equ:complete-lower}
\end{equation}
When $\epsilon\to0$, this becomes
\[
r\geq a\,\frac{1}{\epsilon}\,\frac{\log n}{\binom{n-1}{k-1}}(1+o(1)).
\]
For fixed $k$, the right-hand side is of order $\log n/(\epsilon n^{k-1})$. The sufficient condition~\eqref{equ:hyper-sharp} is of the same order, so the two conditions differ only by a factor depending on $k$ and $a$. For complete graph designs, this necessary condition appears in \citet[Section~3.3.3]{Wang2022Thesis}, which noted that it is smaller by a factor of order $\epsilon^{-1}$ than the concentration-based sufficient condition of Corollary~\ref{thm:eps}; Theorem~\ref{thm:hyper-sharp} and Corollary~\ref{cor:graph-sharp} close this gap.

\paragraph{An incomplete design}
Proposition~\ref{pro:master} applies to any design, while Theorem~\ref{thm:hyper-sharp} and Corollary~\ref{cor:graph-sharp} evaluate it for complete designs. At the other extreme are designs with the smallest possible number of candidate edges. Full column rank requires at least $n$ distinct candidate edges, and for graphs with $n$ odd, this minimum is attained by a cycle through all vertices. Suppose the candidate edges form such an odd cycle, and each candidate edge has replication number $r$. Each vertex then takes part in only $2r$ experiments, compared with $r(n-1)$ in a complete graph design. The incidence matrix of the odd cycle is invertible, so $C_{\bm X}$ is the image of the box $[0,r]^n$ under an invertible linear map, and $\bm d(\bm y)$ lies in its interior exactly when each aggregated count $\tilde y_S$ lies strictly between $0$ and $r$. If the edge probabilities belong to $[\epsilon,1-\epsilon]$, each candidate edge has an extreme count with probability at most $2(1-\epsilon)^r$, and a union bound over the $n$ candidate edges gives failure probability at most $2n(1-\epsilon)^r$. Hence, for any $a>0$, a unique MLE exists with probability at least $1-n^{-a}$ if
\[
r\geq\frac{1}{\log(1/(1-\epsilon))}\,\bigl((1+a)\log n+\log2\bigr).
\]
The required replication grows logarithmically in $n$, whereas for complete graph designs Corollary~\ref{cor:graph-sharp} requires $r$ only of order $\epsilon^{-1}\log n/n$. For an even cycle, the design is bipartite and the vertex parameters are not identifiable. The example illustrates the importance of replication when only a small collection of groups is tested.

The results of this section concern the existence of the MLE. Existence alone does not show that the MLE is close to the true parameter. In the next section, we bound the estimation error for the same fixed designs, in terms of the correlation matrix of the design.

%% file: fig-separation.tex
\begin{tikzpicture}[x=1cm,y=1cm,
  region/.style={draw=black!45,line width=0.4pt},
  vtx/.style={circle,draw=black,fill=white,line width=0.4pt,inner sep=0pt,minimum size=4.2pt},
  succ/.style={draw=black!80,line width=0.55pt,fill=black!12,fill opacity=0.55},
  fail/.style={draw=black!80,line width=0.55pt,dash pattern=on 2.6pt off 1.6pt,fill=black!12,fill opacity=0.55},
  free/.style={draw=black!55,line width=0.45pt,dash pattern=on 0.7pt off 1.3pt}]
\draw[region] (1.6,1.6) ellipse [x radius=1.55, y radius=1.5];
\draw[region] (4.6,1.5) ellipse [x radius=0.85, y radius=1.45];
\draw[region] (7.7,1.6) ellipse [x radius=1.45, y radius=1.5];
\node at (1.6,3.4) {\small $P$: $\alpha_v>0$};
\node at (4.6,3.4) {\small $Z_0$: $\alpha_v=0$};
\node at (7.7,3.4) {\small $N$: $\alpha_v<0$};
\coordinate (a) at (0.8,2.2); \coordinate (b) at (1.6,2.6); \coordinate (c) at (2.3,2.0);
\coordinate (d) at (1.0,0.9); \coordinate (e) at (2.2,0.7);
\coordinate (f) at (4.3,2.4); \coordinate (g) at (4.9,1.5); \coordinate (h) at (4.4,0.6);
\coordinate (i) at (6.9,2.3); \coordinate (j) at (7.7,2.6); \coordinate (k) at (8.3,1.85);
\coordinate (l) at (7.2,1.0); \coordinate (m) at (8.3,0.8);
\draw[succ] (1.55,2.3) ellipse [x radius=1.15, y radius=0.6];
\draw[succ,rotate around={-2:(2.7,0.75)}] (2.7,0.75) ellipse [x radius=2.15, y radius=0.4];
\draw[fail] (6.0,2.45) ellipse [x radius=2.2, y radius=0.42];
\draw[fail] (7.9,1.35) ellipse [x radius=0.95, y radius=0.9];
\draw[free,rotate around={-11.5:(4.75,1.5)}] (4.75,1.5) ellipse [x radius=2.8, y radius=0.4];
\draw[free] (4.55,1.5) ellipse [x radius=0.5, y radius=1.15];
\foreach \v in {a,b,c,d,e,f,g,h,i,j,k,l,m} \node[vtx] at (\v) {};
\begin{scope}[shift={(0.4,-0.75)}]
\draw[succ] (0.0,0) ellipse [x radius=0.3, y radius=0.13]; \node[anchor=west] at (0.4,0) {\small must succeed};
\draw[fail] (3.2,0) ellipse [x radius=0.3, y radius=0.13]; \node[anchor=west] at (3.6,0) {\small must fail};
\draw[free] (5.8,0) ellipse [x radius=0.3, y radius=0.13]; \node[anchor=west] at (6.2,0) {\small not determined by signs alone};
\end{scope}
\end{tikzpicture}

%% file: accuracy.tex
In this section, we examine how the estimation error depends on the design of experiments. A key quantity is the smallest eigenvalue $\lambda_1(\bm M)$ of the correlation matrix. We first give finite-sample bounds for arbitrary designs, then evaluate them for complete designs and relate the matrix condition to graph non-bipartiteness.

\subsection{Error bounds for arbitrary designs}
Let $\Theta_b=\{\bm\beta':\norm{\bm\beta'}_\infty\leq b\}$, where $b>0$, contain the true parameter, and let $\hat{\bm\beta}_b$ maximize the likelihood over $\Theta_b$. For a design with full column rank, this estimate exists and is unique. Whenever the unconstrained MLE exists and belongs to $\Theta_b$, it equals $\hat{\bm\beta}_b$. Write
\[
q_{kb}=\sigma'(kb)=\frac{e^{kb}}{(1+e^{kb})^2}.
\]

\begin{thm}\label{lem:mse1}
Suppose $\norm{\bm x_j}_1\leq k$ for all experiments, $\lambda_1(\bm M)>0$, and $\norm{\bm\beta}_\infty\leq b$. Then $\hat{\bm\beta}_b$ exists and is unique, and
\begin{equation}
\frac1n\E_{\bm\beta}\norm{\hat{\bm\beta}_b-\bm\beta}^2
\leq\frac{1}{4q_{kb}^2\lambda_1(\bm M)}.\label{equ:expected-error}
\end{equation}
For every $\eta\in(0,1)$, with probability at least $1-\eta$,
\begin{equation}
\frac{1}{\sqrt n}\norm{\hat{\bm\beta}_b-\bm\beta}
\leq\frac{1}{q_{kb}\sqrt{\lambda_1(\bm M)}}
\left(\frac12+\sqrt{\frac{\log(1/\eta)}2}\right).\label{equ:l2bound}
\end{equation}
Let $B_b$ be the event that the unconstrained MLE exists and belongs to $\Theta_b$. On $B_b$, $\hat{\bm\beta}=\hat{\bm\beta}_b$, and hence
\[
\P\left[B_b\cap\left\{\frac{1}{\sqrt n}\norm{\hat{\bm\beta}-\bm\beta}
>\hbox{right-hand side of }\eqref{equ:l2bound}\right\}\right]\leq\eta.
\]
\end{thm}

The proof in Appendix~\ref{app:mse1} uses strong convexity in the norm induced by $\bm M$. The score vector, after multiplication by $\bm M^{-1/2}$, has expected squared norm at most $n/4$, yielding~\eqref{equ:expected-error}. A bounded differences inequality gives~\eqref{equ:l2bound}. Bounds of this kind are known for models of paired comparisons and partial rankings, where the role of $\bm M$ is played by a weighted graph Laplacian and that of $\lambda_1(\bm M)$ by its spectral gap \citep{HOX14,SBBPRW16,VY2016}. Theorem~\ref{lem:mse1} gives the analog for the $\beta$-model, whose correlation matrix is not a Laplacian.

\paragraph{Complete designs}
For $2\leq k<n$, let $\bm G_{n,k}$ denote the correlation matrix of the complete $k$-uniform design with one experiment on every $k$-set. Each vertex belongs to $\binom{n-1}{k-1}$ of the $k$-sets and each pair of vertices to $\binom{n-2}{k-2}$ of them. Since $\binom{n-1}{k-1}=\binom{n-2}{k-1}+\binom{n-2}{k-2}$,
\begin{equation}
\bm G_{n,k}=\binom{n-2}{k-1}\bm I+\binom{n-2}{k-2}\bm J,\label{equ:complete-Gram}
\end{equation}
where $\bm J$ is the all-ones matrix. A balanced complete design with replication number $r$ has $\bm M=r\bm G_{n,k}$. The eigenvalues of $\bm G_{n,k}$ are $\binom{n-2}{k-1}$, with multiplicity $n-1$, and $k\binom{n-1}{k-1}$, with multiplicity one. Thus for balanced complete experiments,
\begin{equation}
\lambda_1(\bm M)=r\binom{n-2}{k-1}
=m\frac{k(n-k)}{n(n-1)}.\label{equ:complete-eigenvalue}
\end{equation}
It follows that
\begin{equation}
\frac1n\E\norm{\hat{\bm\beta}_b-\bm\beta}^2
\leq\frac{1}{4q_{kb}^2r\binom{n-2}{k-1}}.\label{equ:complete-error}
\end{equation}
For fixed $k$ and $b$, the expected mean squared error therefore has order at most $1/(rn^{k-1})$. Taking $\eta=1/n$ in~\eqref{equ:l2bound} gives a root mean squared error bound of order $\sqrt{\log n/(rn^{k-1})}$ with probability at least $1-1/n$.

The value~\eqref{equ:complete-eigenvalue} is the largest possible for a given number of experiments.
\begin{pro}\label{pro:lambda-max}
For every $k$-uniform design with $m$ experiments,
\[
\lambda_1(\bm M)\leq m\frac{k(n-k)}{n(n-1)},
\]
with equality if and only if every vertex takes part in the same number of experiments and every pair of distinct vertices takes part together in the same number of experiments.
\end{pro}
Indeed, $\operatorname{tr}(\bm M)=\sum_j\norm{\bm x_j}^2=mk$ and $\bm 1^\top\bm M\bm 1=\sum_j(\bm 1^\top\bm x_j)^2=mk^2$, so the largest eigenvalue of $\bm M$ is at least $mk^2/n$. The remaining $n-1$ eigenvalues sum to at most $mk-mk^2/n$, which gives the bound. Equality requires all these eigenvalues to be equal and $\bm 1$ to be an eigenvector, which means that $\bm M=a\bm I+b\bm J$ for some $a$ and $b$; conversely, every such $\bm M$ attains the bound. For $k\geq2$, the second condition in Proposition~\ref{pro:lambda-max} implies the first: each experiment involving $v$ contains $k-1$ other vertices, so $M_{v,v}=\sum_{u\ne v}M_{u,v}/(k-1)$. Equal replication numbers alone do not suffice. Thus balanced complete designs, and balanced incomplete block designs in which every pair of vertices lies in the same number of candidate edges and all replication numbers are equal, maximize $\lambda_1(\bm M)$, and by Proposition~\ref{lem:lambdaconc}, uniform random designs come within a factor $1-\delta$ of the maximum with high probability.

We complement the upper bound with a lower bound valid for every estimator.
\begin{pro}\label{pro:minimax}
For any design, let $d_{\max}=\max_v M_{v,v}$. Then
\begin{equation}
\inf_{\tilde{\bm\beta}}\sup_{\bm\beta\in\Theta_b}
\frac1n\E_{\bm\beta}\norm{\tilde{\bm\beta}-\bm\beta}^2
\geq\frac14\min\{b^2,1/d_{\max}\},\label{equ:minimax}
\end{equation}
where the infimum is over all estimators based on the outcomes and $1/0$ is interpreted as $+\infty$.
\end{pro}
The proof in Appendix~\ref{app:mse1} is a hypercube argument in the style of Assouad's lemma \citep{Tsybakov2009}. It considers the parameters $\rho\bm\tau$, for $\bm\tau\in\{-1,1\}^n$ and $\rho^2=\min\{b^2,1/d_{\max}\}$. Two such parameters that differ in a single coordinate $v$ have Kullback--Leibler divergence at most $\rho^2M_{v,v}/2\leq1/2$. Hence no estimator can reliably determine the sign of every coordinate, and each coordinate contributes an expected squared error of at least $\rho^2/4$.
For balanced complete experiments, $d_{\max}=r\binom{n-1}{k-1}$. Once $r\geq b^{-2}/\binom{n-1}{k-1}$, the expected mean squared error is bounded below by $1/(4r\binom{n-1}{k-1})$. Since $\binom{n-2}{k-1}/\binom{n-1}{k-1}=(n-k)/(n-1)$, this matches~\eqref{equ:complete-error} in its dependence on $n$ and $r$ for fixed $k$ and $b$. The rate is therefore minimax optimal up to constants, in expected risk. Its dependence on $n$ agrees with the complete-design estimation order of \cite{NandyBhattacharya2024}; the design-matrix bound also quantifies replication through~\eqref{equ:complete-error}.

\begin{ex}\label{ex:rmse-complete}
This example illustrates how the estimation error of complete designs decreases with $n$, at the rate given by Theorem~\ref{lem:mse1} and Proposition~\ref{pro:minimax}. We take one experiment on every $k$-set, for $k=2,3,5$ and $n$ between $k+2$ and $60$. For each pair $(k,n)$, we draw $20$ parameter vectors uniformly from $[-1,1]^n$, each with one sample of the outcomes. The estimate maximizes the likelihood over $\Theta_b$ with $b=2$ and is computed by a quasi-Newton method with box constraints.

Figure~\ref{fig:rmse-c} shows the root mean square error $(n^{-1}\norm{\hat{\bm\beta}_b-\bm\beta}^2)^{1/2}$, averaged over the $20$ samples, against $n$, with $95\%$ confidence intervals. For large $n$, the slopes on the logarithmic scale are close to $-(k-1)/2$, which is the order of $(n/m)^{1/2}$ for $m=\binom nk$. For small $n$, the MLE often fails to exist, and the box constraint keeps the estimate bounded.
\end{ex}
\begin{figure}[t]
\centering
\includegraphics[width=0.5\linewidth]{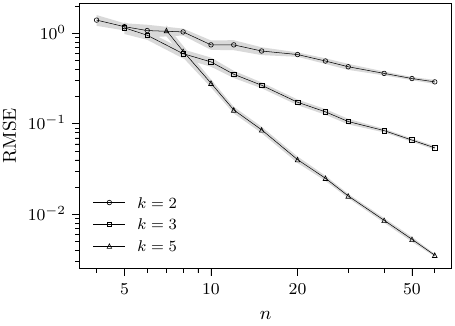}
\caption{Root mean square error versus the number of vertices $n$ for complete $k$-uniform designs, with $95\%$ confidence intervals.}
\label{fig:rmse-c}
\end{figure}

\begin{ex}
This example examines whether $\lambda_1(\bm M)$ governs the estimation error across designs of different structure, as Theorem~\ref{lem:mse1} suggests. We take $n=30$ vertices and $k=3$, draw the parameters uniformly from $[-1/2,1/2]$, and compute $\hat{\bm\beta}_b$ with $b=2$ as in Example~\ref{ex:rmse-complete}. We consider three families of designs:
\begin{itemize}
\item uniform random designs, with $m$ between $150$ and $2400$;
\item designs with few distinct groups: $m=900$ experiments on $G\in\{30,60,150,450\}$ randomly chosen groups, each tested $900/G$ times;
\item skewed designs with $m=900$, in which each group has two vertices from $\{1,\ldots,15\}$ and one from $\{16,\ldots,30\}$, except that, with probability $\delta\in[0.005,0.5]$, it is a uniformly random $3$-set.
\end{itemize}
In a skewed design with small $\delta$, the vector equal to $1$ on $\{1,\ldots,15\}$ and $-2$ on $\{16,\ldots,30\}$ sums to zero on almost every group, so $\lambda_1(\bm M)$ is small.

For each design, Figure~\ref{fig:error-lambda}(a) shows the mean squared error per parameter over $200$ samples of the outcomes against $1/\lambda_1(\bm M)$. Three observations stand out. First, all points lie below $4/\lambda_1(\bm M)$, which is the bound~\eqref{equ:expected-error} with $\sigma'(kb)$ replaced by $\sigma'(0)=1/4$; the bound~\eqref{equ:expected-error} itself is larger by a factor of about $10^4$, because it holds uniformly over $\Theta_b$. Second, for the uniform designs, the error is roughly proportional to $1/\lambda_1(\bm M)$. Third, for the other two families, a single direction is poorly determined, and the error grows more slowly than $1/\lambda_1(\bm M)$: a single small eigenvalue contributes about $1/(n\lambda_1(\bm M))$ to the error per parameter, not $1/\lambda_1(\bm M)$. The error is instead governed by the whole spectrum. Panel~(b) shows the same errors against $n^{-1}\operatorname{tr}(\mathcal I(\bm\beta)^{-1})$, where
\[
\mathcal I(\bm\beta)=\nabla^2(-\ell)(\bm\beta)=\sum_j\sigma'(\bm x_j^\top\bm\beta)\bm x_j\bm x_j^\top
\]
is the Fisher information at the true parameter, by~\eqref{equ:hessian}. In classical asymptotics, $n^{-1}\operatorname{tr}(\mathcal I(\bm\beta)^{-1})$ is the variance per parameter of the MLE; we use it here as a heuristic and do not prove such a statement. On $\Theta_b$, $\mathcal I(\bm\beta)\succeq\sigma'(kb)\bm M$, so $n^{-1}\operatorname{tr}(\mathcal I(\bm\beta)^{-1})\leq n^{-1}\operatorname{tr}(\bm M^{-1})/\sigma'(kb)\leq1/(\sigma'(kb)\lambda_1(\bm M))$; Theorem~\ref{lem:mse1} replaces the average of the inverse eigenvalues of $\bm M$ by the largest one. For all three families, the points lie close to the line $y=x$; the exception is the design with only $30$ distinct groups, where the box constraint limits the error.
\end{ex}
\begin{figure}[t]
\centering
\includegraphics[width=\linewidth]{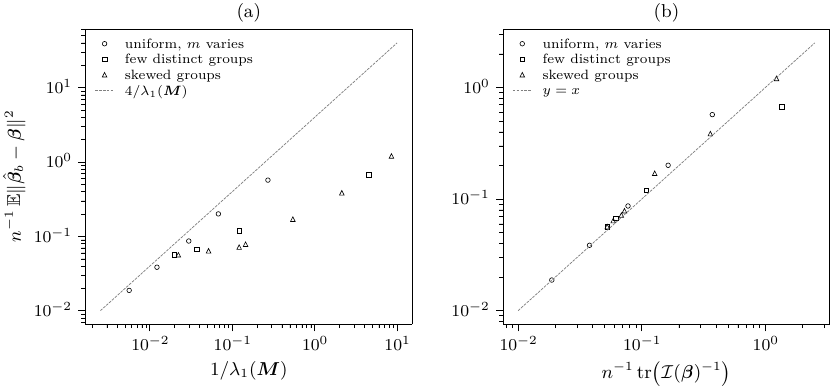}
\caption{Mean squared estimation error per parameter for designs of different structure, with $n=30$ and $k=3$, against (a)~$1/\lambda_1(\bm M)$ and (b)~$n^{-1}\operatorname{tr}(\mathcal I(\bm\beta)^{-1})$, where $\mathcal I(\bm\beta)$ is the Fisher information.}
\label{fig:error-lambda}
\end{figure}

\subsection{The smallest eigenvalue and graph non-bipartiteness}
For any design, $\bm\alpha^\top\bm M\bm\alpha=\norm{\bm X\bm\alpha}^2$, and the Rayleigh--Ritz characterization of the smallest eigenvalue of a symmetric matrix \citep{HornJohnson2013} gives
\begin{equation}
\lambda_1(\bm M)=\min_{\norm{\bm\alpha}=1}
\sum_{j=1}^m\left(\sum_{v\in S_j}\alpha_v\right)^2.\label{equ:variational}
\end{equation}
Consequently, the design has full rank if and only if it admits no nonzero vertex labeling whose labels sum to zero on every candidate edge.

For graph experiments, this condition has a simple combinatorial interpretation. Let $G$ be the multigraph on $V$ with one edge for each experiment. Then $\bm M=\bm D+\bm A$ is its signless Laplacian. On a connected component, the equations $\alpha_u+\alpha_v=0$ force labels to alternate along edges. A nonzero solution exists precisely for a bipartite component, including an isolated vertex. Hence $\lambda_1(\bm M)>0$ exactly when every component is non-bipartite. Along a bipartite component with parts $S,T$, the transformation $\bm\beta\mapsto\bm\beta+t(\ind_S-\ind_T)$ preserves every tested edge probability.

For nonempty $S\subseteq V$, let $\operatorname{cut}(S)$ count the edges between $S$ and its complement, and let $e_{\min}(S)$ be the least number of edges to remove from the induced multigraph on $S$ to make it bipartite. Define
\[
\psi=\min_{\emptyset\ne S\subseteq V}
\frac{|\operatorname{cut}(S)|+e_{\min}(S)}{|S|}.
\]
The parameter measures non-bipartiteness: $\psi=0$ exactly when a component is bipartite. With edge multiplicities counted as weights, \cite{DR94} give
\begin{equation}
\frac{\psi^2}{4d^*}\leq\lambda_1(\bm M)\leq4\psi,\label{equ:dr}
\end{equation}
where $d^*$ is the maximum weighted degree. Further relations are given by \cite{FF12}. Together with Theorem~\ref{lem:mse1}, these inequalities connect estimation accuracy to the extent to which the design is separated from bipartiteness.

For hypergraphs, the zero-sum labeling condition~\eqref{equ:variational} determines identifiability.

\begin{remark}
Its relation to the projected graph follows from the decomposition below. Put $A_{u,v}=M_{u,v}$ for $u\ne v$, $A_{u,u}=0$, $N_{u,u}=M_{u,u}$, and $D_{u,u}=\sum_{v\ne u}M_{u,v}$. Then
\[
\bm M=(\bm D+\bm A)-(\bm D-\bm N).
\]
Here $\bm D+\bm A$ is the signless Laplacian of the projected graph on $V$, in which the pair $\{u,v\}$ has weight $M_{u,v}$, and its maximum weighted degree is $\max_uD_{u,u}$. For $k$-uniform experiments with $k\geq2$, each experiment involving $u$ contributes $k-1$ to $D_{u,u}$, so $\bm D=(k-1)\bm N$ and $\bm M=(\bm D+\bm A)-(k-2)\bm N$. Let $\psi$ denote the weighted non-bipartiteness parameter of the projected graph. Applying~\eqref{equ:dr} to $\bm D+\bm A$, with $d^*=(k-1)\max_vM_{v,v}$, and using that adding the diagonal matrix $-(k-2)\bm N$ shifts the smallest eigenvalue by at least $-(k-2)\max_vM_{v,v}$ and at most $-(k-2)\min_vM_{v,v}$ \citep{HornJohnson2013}, we obtain
\begin{equation}
\frac{\psi^2}{4(k-1)\max_vM_{v,v}}-(k-2)\max_vM_{v,v}
\leq\lambda_1(\bm M)
\leq4\psi-(k-2)\min_vM_{v,v}.\label{equ:hyper-shift}
\end{equation}
Taking a singleton in the definition of $\psi$ gives $\psi\leq(k-1)\max_vM_{v,v}$. For $k\geq3$, the lower bound in~\eqref{equ:hyper-shift} is therefore nonpositive and does not establish full rank.
\end{remark}

For graph designs, \eqref{equ:dr} relates $\lambda_1(\bm M)$ to the structure of the design. For hypergraph designs with $k\geq3$, we have no comparable combinatorial lower bound. In the next section, we show that for random designs $\lambda_1(\bm M)$ is close to the smallest eigenvalue of its expectation, which gives explicit error bounds.

%% file: random.tex
In this section, the design is itself random. Random designs model situations in which the tested groups are not chosen by the analyst, as when groups are sampled at random in crowdsourcing or controlled experiments. They also serve as a benchmark for incomplete designs: how many experiments suffice for identifiability, existence, and accuracy? Each of the $m$ experiments tests a $k$-subset of $V$ drawn uniformly at random, with replacement, independently of the other experiments. Equivalently, the rows of the design matrix are drawn independently and uniformly from $\{\bm x\in\{0,1\}^n:\norm{\bm x}_1=k\}$. The candidate edges and their replication numbers are therefore random. Conditional on the design, the outcomes follow the $\beta$-model independently. There are thus two sources of randomness: the draw of the design and the outcomes given the design. Section~\ref{sec:random-rank} concerns the rank of the design matrix, which depends on the design alone. Sections~\ref{sec:random-existence} and~\ref{sec:random-error} concern MLE existence and estimation accuracy, and their probabilities are over both sources.

\subsection{Rank}\label{sec:random-rank}
By Lemma~\ref{lem:sc} and Proposition~\ref{thm:mle-overlap}, a unique MLE can exist only if the design matrix has full column rank. Equivalently, the vertex parameters must be identifiable from the success probabilities of the tested groups. Full rank is also needed for the error bound of Theorem~\ref{lem:mse1}, which requires $\lambda_1(\bm M)>0$. For a fixed design, full rank is a deterministic property, characterized in Section~\ref{sec:accuracy}. For a random design, it is a random event that depends only on the draw of the groups, not on the parameter or the outcomes. Its threshold is therefore the smallest number of experiments with which the parameters can be estimated at all, whatever the edge probabilities. In terms of hypergraphs, $\bm X$ is the incidence matrix of the random multi-hypergraph formed by the tested groups, and the question is when this matrix has full column rank.

Comparing the rank threshold with the existence threshold of Section~\ref{sec:random-existence} separates two costs. Full rank requires $m$ of order $\max\{n,(n/k)\log n\}$, so that, in particular, every vertex takes part in some experiment. For fixed $k$, a unique MLE requires $m$ of order $\epsilon^{-1}(n/k)\log n$, so that, in particular, every vertex has both a success and a failure. The factor $\epsilon^{-1}$ is thus the cost of the outcomes, beyond identifiability.

A necessary condition for full rank is that the number of experiments is at least the number of vertices. Another is that every vertex takes part in at least one experiment. For a fixed vertex, the probability of a null column is $(1-k/n)^m$. Thus
\begin{equation}
\P[\bm X\text{ has a null column}]\leq n e^{-km/n}.\label{equ:null-column-upper}
\end{equation}
In particular, $m\geq(1+a)(n/k)\log n$ ensures no null column with probability at least $1-n^{-a}$. This bound, and the fact that its order $(n/k)\log n$ cannot be lowered when $k=o(n/\log n)$, appear in \citet[Theorem~4.2.2]{Wang2022Thesis}.

We give two sufficient conditions for full rank, which differ in the range of $k$ and in their precision. Theorem~\ref{thm:rank-growing} allows the group size $k$ to grow with $n$, up to $o(n/\log n)$. For groups of size at least $(1+\xi)\log n$, for a fixed $\xi>0$, it shows that $m=n$ experiments suffice, which is the smallest possible number. For smaller groups, it shows that $m$ of order $(n/k)\log n$ suffices, which is the necessary order, but with a constant $c>2$ multiplying $(n/k)\log n$ in place of the threshold value $1$. In both cases the failure probability tends to zero slowly. Theorem~\ref{thm:rank-fixed} is restricted to fixed $k$ and is sharp. Full rank holds with high probability for $m=\lfloor c(n/k)\log n\rfloor$ with $c>1$ and fails with probability tending to one for $c<1$, and above the threshold the failure probability is asymptotically the probability of a null column. For fixed $k$, Theorem~\ref{thm:rank-fixed} therefore improves Theorem~\ref{thm:rank-growing}, while Theorem~\ref{thm:rank-growing} also covers groups whose size grows with $n$.

\Needspace{6\baselineskip}
\begin{thm}\label{thm:rank-growing}
Assume that $2\leq k=k(n)=o(n/\log n)$ and that, for a fixed constant $c>2$,
\begin{equation}
m\geq\max\{n,c(n/k)\log n\}.\label{equ:rank-growing}
\end{equation}
Then
\[
\P[\rank_{\reals}(\bm X)<n]
=O_c\left(\frac{(\log\log n)^{3/2}}{\sqrt{\log n}}\right).
\]
If, moreover, $k\geq(1+\xi)\log n$ for a fixed $\xi>0$, the same bound, with a constant depending on $\xi$, holds under $m\geq n$ alone. Conversely, $m\geq n$ is necessary for full column rank, and if $m\leq c'(n/k)\log n$ for a fixed $c'<1$, the probability of full column rank tends to zero.
\end{thm}

Theorem~\ref{thm:rank-growing} proves Conjecture~4.2.3 in \citet{Wang2022Thesis} in a stronger form. The conjecture makes the same assumptions and asserts a failure probability of order $(\log\log n)^{-1/4}$, which is larger than the bound of Theorem~\ref{thm:rank-growing}. The thesis outlined an approach to the conjecture based on the vertex-exposure process of \cite{Costello2008}, which it reduced to a further unproved conjecture \citep[Conjecture~4.2.12]{Wang2022Thesis}. The proof here uses a different argument. The proof in Appendix~\ref{app:rank-growing} reduces the problem to square matrices whose rows are independent uniform incidence vectors of $d$-sets. \cite{FKS22} showed that such a matrix is nonsingular with high probability once $(1+\xi)\log n\leq d\leq n/2$, for a fixed $\xi>0$. Lemma~\ref{lm:quantitative-rank} is a quantitative version of their result: the singularity probability is of order $n^{-\gamma}\log d+(\log d)^{3/2}/\sqrt d$, for a constant $\gamma>0$, uniformly in $d$.

The proof considers two cases. If $k\geq(1+\xi)\log n$, the first $n$ rows of $\bm X$ already form such a matrix with $d=k$. If $k$ is smaller, the rows of $\bm X$ are too sparse, and the proof builds denser rows from sums of rows. It divides the experiments into $(1+\zeta)n$ disjoint blocks of $\nu =\lceil(1+\xi)\log n/k\rceil$ rows each, for a fixed $\zeta>0$. A block is successful when its $\nu $ candidate edges are pairwise disjoint. The sum of its rows is then the incidence vector of a uniformly random set of size $d=\nu k$, which is of order $\log n$. A block fails with probability $O((\log n)^2/n)$, so with high probability at least $n$ blocks succeed. The sums of $n$ successful blocks form an $n\times n$ matrix with independent uniform rows of weight $d$, to which Lemma~\ref{lm:quantitative-rank} applies. Each of these rows is a sum of rows of $\bm X$, so if this matrix is nonsingular, then $\bm X$ has full column rank. With $d$ of order $\log n$, the second term of the lemma gives the rate in Theorem~\ref{thm:rank-growing}. The blocks use at most $2(1+\xi)(1+\zeta)(n/k)\log n$ rows, because $\nu \leq2(1+\xi)\log n/k$, and this is where the constant $c>2$ comes from. The factor $2$ is caused by rounding $\nu $ up to an integer, and it matters only when $k$ is of the order $\log n$: for $k=o(\log n)$, we have $\nu =(1+o(1))(1+\xi)\log n/k$, and the same argument gives the conclusion for every fixed $c>1$. For the converse, the number of null columns has mean $n(1-k/n)^m$, which tends to infinity when $m\leq c'(n/k)\log n$ with $c'<1$. The null-column indicators are nonpositively correlated, so Chebyshev's inequality shows that a null column exists with high probability.

The sufficient condition~\eqref{equ:rank-growing} and the necessary conditions of Theorem~\ref{thm:rank-growing}, namely $m\geq n$ and $m>c'(n/k)\log n$ for every fixed $c'<1$, are of the same order $\max\{n,(n/k)\log n\}$. For $k\geq(1+\xi)\log n$ they coincide, and otherwise they differ only in the constant multiplying $(n/k)\log n$. For fixed $k$, the next theorem closes this gap.
\Needspace{12\baselineskip}
\begin{thm}\label{thm:rank-fixed}
Assume that $k\geq2$ and $c>0$ are fixed, and that $m=\lfloor c(n/k)\log n\rfloor$. If $c>1$, then
\begin{equation}
\P[\rank_{\reals}(\bm X)<n]=(1+o(1))n^{-(c-1)}.\label{equ:rank-fixed}
\end{equation}
If $c<1$, this probability tends to one. Consequently, for any fixed $a>0$ and $c>1+a$, $m\geq c(n/k)\log n$ ensures full rank with probability at least $1-n^{-a}$ for all sufficiently large $n$.
\end{thm}
Like the arguments of \cite{CFP19} and \cite{FKS22}, the proof in Appendix~\ref{app:rank-fixed} works over a finite field. It bounds the probability that $\bm X$ has a nonzero kernel vector, with separate arguments for kernel vectors of small and of large support. The proof works over the field $\mathbb F_p$ of integers modulo a prime $p>k$. Full rank over $\mathbb F_p$ implies full rank over the real numbers, because the $n\times n$ minors of $\bm X$ are integers, and a minor that is nonzero modulo $p$ is nonzero. Working over a finite field makes the number of candidate kernel vectors finite, and $p>k$ ensures that $ka\ne0$ in $\mathbb F_p$ for every nonzero $a$, which is used for kernel vectors of large support.

Consider a kernel vector with support $S$ of size $s$. Every row of $\bm X$ that meets $S$ in exactly one vertex has a nonzero product with this vector. Such a row occurs with probability $a_s\approx ks/n$, so a kernel vector with support $S$ exists with probability at most $(1-a_s)^m$. This bound does not depend on the labels on $S$, and a union bound over the $\binom ns$ supports suffices. For $s=1$, the bound is the null-column probability $n(1-k/n)^m=(1+o(1))n^{-(c-1)}$. Supports with $2\leq s\leq n/(\log n)^2$ contribute $O(n^{-2(c-1)})$, and the remaining supports with $s\leq\delta n$, for a small fixed $\delta>0$, contribute $\exp\{-\Omega(n/\log n)\}$. For $s>\delta n$, some nonzero label appears on at least $\delta n/(p-1)$ vertices. A row whose $k$ vertices all carry this label $a$ has product $ka$, which is nonzero because $p>k$. Such a row occurs with probability bounded below by a positive constant, which outweighs the at most $p^n$ labeled vectors. Hence, above the threshold, null columns account for rank deficiency asymptotically. A second-moment argument for the number of null columns gives the matching lower bound and the behavior for $c<1$.

\begin{ex}
This example illustrates the rank threshold of Theorem~\ref{thm:rank-fixed} and the role of null columns. For $n=200$ and $k=2,3,5$, we sampled $250$ uniform random designs for each $m=\lfloor c(n/k)\log n\rfloor$, with $c$ between $0.5$ and $2.5$, and recorded whether $\bm X$ has full column rank and whether it has a null column.

Figure~\ref{fig:rank-threshold} shows the frequencies of the two events. They coincide, except for $k=5$ and $c\leq0.875$, where $m<n$ and $\bm X$ is rank deficient for that reason alone. Both follow $1-\exp(-n^{-(c-1)})$, the Poisson approximation to the probability of a null column, with the transition at $c=1$ predicted by Theorem~\ref{thm:rank-fixed}. For larger groups, $k=10,20,40$, which exceed $\log n\approx5.3$, every sampled design with $m=n$ had full rank, in agreement with the sufficiency statement of Theorem~\ref{thm:rank-growing}; these results are not shown in the figure. Similar experiments for $n=100$, for the null-column event and for full rank separately, are reported in \citet[Figures~4.1 and~4.2]{Wang2022Thesis}.
\end{ex}
\begin{figure}[t]
\centering
\includegraphics[width=\linewidth]{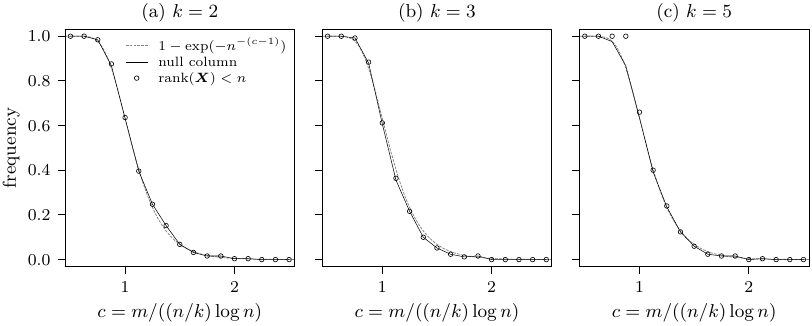}
\caption{Frequency of the event $\rank(\bm X)<n$ (markers) and of a null column (solid line) for uniform random designs with $n=200$ and (a)~$k=2$, (b)~$k=3$, (c)~$k=5$, versus $c=m/((n/k)\log n)$. The dashed curve is $1-\exp(-n^{-(c-1)})$.}
\label{fig:rank-threshold}
\end{figure}

Full rank ensures identifiability, but finite MLE existence also depends on the outcomes, which we consider next.

\subsection{MLE existence and uniqueness}\label{sec:random-existence}
We apply Proposition~\ref{pro:master} directly to the random experiments.
\begin{thm}\label{thm:suffrnd}
Assume $n\geq2k$ and $\epsilon\leq p_S(\bm\beta)\leq1-\epsilon$ for every $k$-set $S$, where $\epsilon\in(0,1/2]$. For any $a>0$, a unique MLE exists with probability at least $1-n^{-a}$ if
\begin{equation}
m\geq\frac{1}{\epsilon}\,\frac{k\binom nk\bigl((1+a)\log n+3\bigr)}{\binom{\lfloor n/2\rfloor-1}{k-1}}.\label{equ:suffrnd}
\end{equation}
For graphs with $n\geq40$, it suffices that
\begin{equation}
m\geq\frac{1}{\epsilon}\,n\bigl((1+a)\log n+3\bigr).\label{equ:suffrnd-graph}
\end{equation}
\end{thm}

The proof in Appendix~\ref{app:suffrnd} applies Proposition~\ref{pro:master} to the random design. For fixed disjoint $P,N$, the number of sampled experiments counted by $W(P,N)$ is binomial. Taking its probability generating function in~\eqref{equ:master} yields an exponential bound, followed by the same counting argument as for complete designs.

For fixed $k$, keeping only the leading term in $n$ of its right-hand side, condition~\eqref{equ:suffrnd} reads
\[
m\geq(1+a)\,\frac{1}{\epsilon}\,2^{k-1}n\log n(1+o(1)).
\]
For fixed $k$, this dependence on $n$ and $\epsilon$ cannot be improved, as the following necessary condition shows. For homogeneous edge probability $\epsilon$, the probability that every experiment involving a fixed vertex fails is $(1-k\epsilon/n)^m$. Since this event gives a separating vector, failure probability at most $n^{-a}$ requires the necessary condition
\begin{equation}
m\geq a\,\frac{1}{\log(1/(1-k\epsilon/n))}\,\log n
=a\,\frac{1}{\epsilon}\,\frac{n\log n}{k}(1+o(1))\label{equ:rndlb}
\end{equation}
when $k\epsilon/n\to0$. Thus for fixed $k$, $m$ of order $\epsilon^{-1}n\log n$ is both sufficient and necessary up to constants for a uniform guarantee over this parameter class. For graphs, this necessary condition appears in \citet[Section~4.3.1]{Wang2022Thesis}, where the corresponding sufficient condition, obtained through degree polytopes, requires $m$ of order $\epsilon^{-1}n^{5/4}(\log n)^{1/4}$.

\begin{ex}
This example illustrates the existence threshold for uniform random designs and its dependence on $\epsilon$, in relation to Theorem~\ref{thm:suffrnd} and the necessary condition~\eqref{equ:rndlb}. We take every edge probability equal to $\epsilon$, as in the derivation of~\eqref{equ:rndlb}, and measure the number of experiments by $x=mk\epsilon/(n\log n)$, in terms of which~\eqref{equ:rndlb} reads $x\geq a(1+o(1))$. We consider $k=2$ with $n=50,100,200$ and $k=3$ with $n=50,100$, each with $\epsilon=0.05,0.1,0.2,0.4$. For each configuration and each value of $x$, we sampled $200$ designs and outcomes and decided exactly whether a unique MLE exists: we checked that $\bm X$ has full rank and, by a linear program, that no separating vector exists.

In Figure~\ref{fig:exist-threshold}(a), for $k=2$ and $n=100$, the curves for the four values of $\epsilon$ coincide and follow $\exp(-n^{-(x-1)})$, the approximate probability that no vertex fails in all its experiments. Panel~(b) shows the same data against $mk\epsilon^2/(n\log n)$; there the curves separate, so the number of experiments needed scales as $\epsilon^{-1}$ rather than $\epsilon^{-2}$. Panel~(c) shows, for $\epsilon=0.1$, the same transition at $x\approx1$ for $k=2,3$ and $n$ between $50$ and $200$. The observed threshold therefore agrees with the necessary condition~\eqref{equ:rndlb}. The sufficient condition~\eqref{equ:suffrnd}, for failure probability at most $n^{-a}$, requires $x\geq(1+a)k2^{k-1}$ to leading order, so its constant is conservative.
\end{ex}
\begin{figure}[t]
\centering
\includegraphics[width=\linewidth]{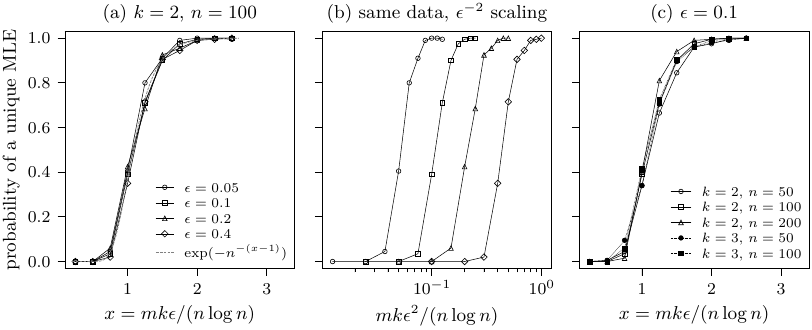}
\caption{Frequency of a unique MLE for uniform random designs with every edge probability equal to $\epsilon$. (a)~$k=2$, $n=100$, versus $x=mk\epsilon/(n\log n)$. (b)~The same data versus $mk\epsilon^2/(n\log n)$. (c)~$\epsilon=0.1$ and several $k$ and $n$, versus $x$.}
\label{fig:exist-threshold}
\end{figure}

\subsection{Estimation error}\label{sec:random-error}
By Theorem~\ref{lem:mse1}, the estimation error is controlled by the smallest eigenvalue $\lambda_1(\bm M)$ of the correlation matrix. Full rank, studied in Section~\ref{sec:random-rank}, ensures only that $\lambda_1(\bm M)>0$, and an error bound requires a lower bound on its magnitude. We therefore compare $\lambda_1(\bm M)$ with the smallest eigenvalue of its expectation. For $u\ne v$,
\[
\E[M_{u,v}]=\frac{k(k-1)}{n(n-1)}m,\qquad
\E[M_{u,u}]=\frac{k}{n}m.
\]
Therefore
\begin{equation}
\lambda_1(\E\bm M)=\frac{k(n-k)}{n(n-1)}m.\label{equ:lam1em}
\end{equation}
The following finite-sample bound, which appears as \citet[Proposition~4.2.1]{Wang2022Thesis}, shows that $\lambda_1(\bm M)$ concentrates around $\lambda_1(\E\bm M)$.
\begin{pro}\label{lem:lambdaconc}
For $2\leq k<n$ and $\delta\in(0,1)$,
\[
\P\left[\lambda_1(\bm M)\leq(1-\delta)\lambda_1(\E\bm M)\right]
\leq n\exp\left(-\frac{\delta^2m(n-k)}{2n(n-1)}\right).
\]
In particular, the failure probability is at most $n^{-a}$ if
\[
m\geq2(a+1)\,\frac{1}{\delta^2}\,\frac{n(n-1)\log n}{n-k}.
\]
\end{pro}
The bound is in effect two-sided. By Proposition~\ref{pro:lambda-max}, every $k$-uniform design with $m$ experiments has $\lambda_1(\bm M)\leq mk(n-k)/(n(n-1))$, which equals $\lambda_1(\E\bm M)$ in~\eqref{equ:lam1em}. Hence $\lambda_1(\bm M)\leq\lambda_1(\E\bm M)$ for every realization of the random design, and, with the probability given in Proposition~\ref{lem:lambdaconc},
\[
(1-\delta)\lambda_1(\E\bm M)\leq\lambda_1(\bm M)\leq\lambda_1(\E\bm M).
\]
Thus uniform random designs nearly attain the largest value of $\lambda_1(\bm M)$ possible for a $k$-uniform design with $m$ experiments.

The proof in Appendix~\ref{app:suffrnd} applies the matrix Chernoff inequality \citep{T15} to $\bm M=\sum_j\bm x_j\bm x_j^\top$, a sum of independent positive semidefinite matrices of spectral norm $k$, and substitutes~\eqref{equ:lam1em}.
Proposition~\ref{lem:lambdaconc} also implies full rank, but only for $m$ of order $n\log n$. This exceeds the order $\max\{n,(n/k)\log n\}$ of Theorems~\ref{thm:rank-growing} and~\ref{thm:rank-fixed} by a factor of order $\min\{k,\log n\}$.

Combining Theorem~\ref{lem:mse1} with Proposition~\ref{lem:lambdaconc} gives an error bound for random designs. Fix $\eta\in(0,1)$, the confidence parameter in~\eqref{equ:l2bound}, and suppose that the sample-size condition of Proposition~\ref{lem:lambdaconc} holds. With probability at least $1-n^{-a}$, the design satisfies $\lambda_1(\bm M)\geq(1-\delta)\lambda_1(\E\bm M)$. Conditional on such a design, \eqref{equ:l2bound} holds with probability at least $1-\eta$ over the outcomes. Substituting~\eqref{equ:lam1em} shows that, with probability at least $1-n^{-a}-\eta$ over both the design and the outcomes,
\begin{equation}
\frac1{\sqrt n}\norm{\hat{\bm\beta}_b-\bm\beta}
\leq\frac1{q_{kb}}
\sqrt{\frac{n(n-1)}{(1-\delta)mk(n-k)}}
\left(\frac12+\sqrt{\frac{\log(1/\eta)}2}\right).\label{equ:rndub}
\end{equation}
Here $\norm{\bm\beta}_\infty\leq b$, as in Theorem~\ref{lem:mse1}. Choose any box-constrained maximizer when the design is rank deficient. Applying~\eqref{equ:expected-error} conditional on the eigenvalue event, and using $\norm{\hat{\bm\beta}_b-\bm\beta}^2/n\leq4b^2$ on its complement, gives
\[
\E\left[\frac1n\norm{\hat{\bm\beta}_b-\bm\beta}^2\right]
\leq\frac{n(n-1)}{4q_{kb}^2(1-\delta)mk(n-k)}+4b^2n^{-a}.
\]
For fixed $k,b$, $a>0$, and $\delta\in(0,1)$, the sample-size condition of Proposition~\ref{lem:lambdaconc} therefore ensures that the normalized error converges to zero in mean square, and hence in probability. The exact factor $(n-k)/(n-1)$ records the loss of information as the group size approaches $n$.

\begin{ex}
This example illustrates how the group size affects the estimation error of uniform random designs, through the smallest expected eigenvalue~\eqref{equ:lam1em}. It repeats Example~4.3.1 of \cite{Wang2022Thesis}, with the box-constrained estimate. We take $n=10,20,40$ vertices, half with parameter $-1$ and half with parameter $1$, and fix $m=n(n-1)/2$. For each $k$ between $2$ and $n-1$, we sample $100$ designs, each with a sample of the outcomes, and compute the estimate as in Example~\ref{ex:rmse-complete}.

Figure~\ref{fig:rmse} shows the root mean square error against $k/n$, with $95\%$ confidence intervals. The error is smallest for moderate $k/n$. It increases for small $k$, where each vertex takes part in fewer experiments, and for $k$ close to $n$, where $\lambda_1(\E\bm M)$ in~\eqref{equ:lam1em}, which is proportional to $k(n-k)$, decreases.
\end{ex}
\begin{figure}[t]
\centering
\includegraphics[width=\linewidth]{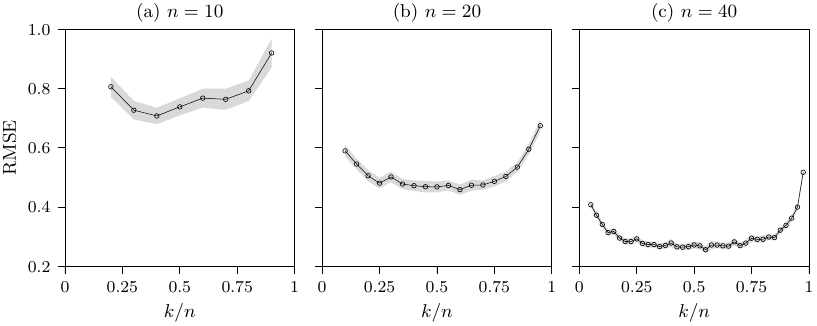}
\caption{Root mean square error versus $k/n$ for uniform random designs with $m=n(n-1)/2$, with $95\%$ confidence intervals: (a)~$n=10$, (b)~$n=20$, (c)~$n=40$.}
\label{fig:rmse}
\end{figure}

%% file: conc.tex
In this paper, we studied maximum likelihood estimation in the $\beta$-model for random hypergraphs under complete, arbitrary, and random designs. We obtained explicit finite-sample conditions for MLE existence and uniqueness, through degree polytopes, interior representations, and separating vectors. For balanced complete designs and uniform random designs, the conditions obtained through separating vectors are linear in $\epsilon^{-1}$ and, for fixed $k$, match the necessary conditions~\eqref{equ:complete-lower} and~\eqref{equ:rndlb} up to constants. We bounded the estimation error for arbitrary designs in terms of the smallest eigenvalue of the correlation matrix, showed that the resulting rate is minimax optimal for balanced complete designs, and showed that uniform random designs nearly maximize this eigenvalue. Finally, we established full-rank conditions for random designs, with the sharp threshold for fixed $k$.

Several questions remain open. First, the numerical results suggest that the constants in our existence conditions can be improved. For uniform random designs with homogeneous edge probabilities, a unique MLE appears to exist with probability tending to one once $m\geq(1+\delta)\epsilon^{-1}(n/k)\log n$, for any fixed $\delta>0$, which is where, with high probability, no vertex fails in all its experiments; Theorem~\ref{thm:suffrnd} requires a constant of order $k2^{k-1}$ in place of $1$. For the rank, the factor $2$ in Theorem~\ref{thm:rank-growing} matters only when $k$ is of the order $\log n$, and removing it there is open. Second, our conditions are uniform over all parameters with edge probabilities in $[\epsilon,1-\epsilon]$; conditions that depend on the parameter may be sharper for heterogeneous parameters. Third, the error bound of Theorem~\ref{lem:mse1} depends on the design only through $\lambda_1(\bm M)$ and on the parameter box through $\sigma'(kb)$; bounds in terms of the whole spectrum of $\bm M$ and the curvature at the true parameter, together with limit theory for incomplete designs, would describe the error more precisely. Finally, by Proposition~\ref{pro:lambda-max}, designs in which every pair of vertices takes part together in the same number of experiments maximize $\lambda_1(\bm M)$. Designs that are optimal for criteria depending on the parameter or for MLE existence, and designs chosen adaptively from earlier outcomes or formed by self-selection of the vertices, remain to be studied.

%% file: appendix.tex
\section{Preliminary results}\label{app:prelim}
\subsection{Proof of Proposition~\ref{thm:mle-overlap}}
We give a direct proof. Write $f=-\ell$ and $\bm a_j=(2y_j-1)\bm x_j$, so
\[
f(\bm\beta)=\sum_j\log(1+e^{-\bm a_j^\top\bm\beta}).
\]
If $\bm\alpha\ne0$ separates the outcomes, $\bm a_j^\top\bm\alpha\geq0$ for every $j$. The function $f(\bm\beta+t\bm\alpha)$ is nonincreasing in $t$. If any inequality is strict, the function decreases strictly, so it has no minimizer. If every product vanishes, it is constant along a line, and a minimizer is nonunique.

Conversely, under overlap, every unit vector $\bm\alpha$ has $\bm a_j^\top\bm\alpha<0$ for some $j$. The continuous function
\[
h(\bm\alpha)=\max_j\{-\bm a_j^\top\bm\alpha\}
\]
is positive on the unit sphere, and hence its minimum $h_*>0$. Using $\log(1+e^z)\geq\max\{0,z\}$ gives $f(\bm\beta)\geq h_*\norm{\bm\beta}$. Thus $f$ is coercive, that is, $f(\bm\beta)\to\infty$ as $\norm{\bm\beta}\to\infty$. A continuous coercive function attains its minimum, because its sublevel sets are bounded and closed. Overlap implies full rank, so~\eqref{equ:hessian} gives strict convexity and uniqueness.

\subsection{The convex-support criterion without full rank}
For completeness, we also derive the relative-interior criterion of Section~\ref{sec:model} when $\bm X$ need not have full column rank. Let $\ker\bm X=\{\bm\alpha\in\reals^n:\bm X\bm\alpha=\bm0\}$ be the kernel of $\bm X$. The likelihood depends on $\bm\beta$ only through $\bm X\bm\beta$, so it is constant along $\ker\bm X$, and it suffices to consider $\bm\beta$ in the orthogonal complement $(\ker\bm X)^\perp$. A proper supporting face of $C_{\bm X}$ through $\bm d(\bm y)$ corresponds to a separating vector whose products with the rows are not all zero. If there is no such face, the argument above, applied to unit vectors in $(\ker\bm X)^\perp$, shows that $f$ is coercive on this subspace. The maximizers in $\reals^n$ are then a single point of $(\ker\bm X)^\perp$ plus $\ker\bm X$.

\section{Graph degree-polytope condition}\label{app:mlesuff}
\subsection{Proof of Theorem~\ref{thm:mlesuff}}
The proof has two steps. The first is deterministic: since the inequalities defining $P_n$ are affine in the degree sequence, the normalized degree sequence can reach the boundary of $P_n$ only if some coordinate deviates from its mean by at least $(n-1)\mathcal E(\E\tilde{\bm d})$. The second is probabilistic: Hoeffding's inequality and a union bound over the vertices show that, under~\eqref{equ:edcond}, such a deviation is unlikely.

By~\eqref{equ:fst}, $f(S,T,\tilde{\bm d},n)=|S|(n-1-|T|)-\sum_{v\in S}\tilde d_v+\sum_{v\in T}\tilde d_v$ is an affine function of $\tilde{\bm d}$, in which $\tilde d_v$ has coefficient $-1$ for $v\in S$, coefficient $1$ for $v\in T$, and coefficient $0$ otherwise. Hence, for every $(S,T)\in\Omega$,
\[
f(S,T,\tilde{\bm d},n)-f(S,T,\E\tilde{\bm d},n)
=-\sum_{v\in S}(\tilde d_v-\E \tilde d_v)+\sum_{v\in T}(\tilde d_v-\E \tilde d_v),
\]
and therefore
\begin{equation}
|f(S,T,\tilde{\bm d},n)-f(S,T,\E\tilde{\bm d},n)|
\leq|S\cup T|\norm{\tilde{\bm d}-\E\tilde{\bm d}}_\infty.\label{equ:fdiff}
\end{equation}

The normalized degree sequence $\tilde{\bm d}$ always lies in $P_n$, since $P_n=Z(E)$ for complete graph designs. Suppose that $\tilde{\bm d}\notin\Int P_n$. By~\eqref{equ:eglinear}, one of the inequalities defining $P_n$ then holds with equality at $\tilde{\bm d}$. If this is the inequality for $(S,T)\in\Omega$, then $f(S,T,\tilde{\bm d},n)=0$, and~\eqref{equ:fdiff} and the definition~\eqref{equ:margin} of $\mathcal E$ give
\[
|S\cup T|\norm{\tilde{\bm d}-\E\tilde{\bm d}}_\infty
\geq f(S,T,\E\tilde{\bm d},n)
\geq|S\cup T|(n-1)\mathcal E(\E\tilde{\bm d}).
\]
If it is a coordinate inequality, for example $\tilde d_v=0$, then $\norm{\tilde{\bm d}-\E\tilde{\bm d}}_\infty\geq\E \tilde d_v\geq(n-1)\mathcal E(\E\tilde{\bm d})$, again by~\eqref{equ:margin}, and the case $\tilde d_v=n-1$ is the same. In every case,
\begin{equation}
\norm{\tilde{\bm d}-\E\tilde{\bm d}}_\infty
\geq(n-1)\mathcal E(\E\tilde{\bm d}).\label{equ:dev-margin}
\end{equation}

We now bound the probability of~\eqref{equ:dev-margin}. For a fixed vertex $u$,
\[
\tilde d_u=\sum_{v\ne u}\frac{\tilde y_{\{u,v\}}}{M_{u,v}}
=\sum_{v\ne u}\sum_{j:S_j=\{u,v\}}\frac{Y_j}{M_{u,v}},
\]
which is a sum of independent random variables, one for each experiment involving $u$. The term of an experiment on $\{u,v\}$ takes values in $\{0,1/M_{u,v}\}$, an interval of length $1/M_{u,v}$. Hoeffding's inequality \citep{hoeffding} states that a sum of independent random variables taking values in intervals of lengths $c_1,c_2,\ldots$ deviates from its mean by at least $x$ with probability at most $2\exp(-2x^2/\sum_jc_j^2)$. Here there are $M_{u,v}$ terms of length $1/M_{u,v}$ for each $v\ne u$, so
\[
\sum_jc_j^2=\sum_{v\ne u}M_{u,v}\frac1{M_{u,v}^2}=\sum_{v\ne u}\frac1{M_{u,v}}=\frac{n-1}{H_u(\bm M)}\leq\frac{n-1}{H(\bm M)}.
\]
A union bound over the $n$ vertices gives
\[
\P[\norm{\tilde{\bm d}-\E\tilde{\bm d}}_\infty\geq x]
\leq2n\exp\left(-\frac{2H(\bm M)x^2}{n-1}\right).
\]
At $x=(n-1)\mathcal E(\E\tilde{\bm d})$, the exponent equals $-2H(\bm M)(n-1)\mathcal E(\E\tilde{\bm d})^2$, which is at most $-2c\log n$ by condition~\eqref{equ:edcond}, so the probability is at most $2n^{1-2c}$. By~\eqref{equ:dev-margin}, $\tilde{\bm d}\in\Int P_n$ with probability at least $1-2n^{1-2c}$, and on this event the interior criterion of Section~\ref{sec:model} gives a unique MLE. For a balanced design, $H(\bm M)=r$ and $m=r\binom n2$, and~\eqref{equ:edcond} becomes $m\geq(c/2)\,\mathcal E(\E[\tilde{\bm d}])^{-2}\,n\log n$.

\subsection{Proof of Corollary~\ref{thm:eps}}
We show that $\mathcal E(\E\tilde{\bm d})\geq\epsilon/4$ and bound each term in~\eqref{equ:margin}. For the coordinate terms, $\E \tilde d_v=\sum_{u\ne v}p_{u,v}\geq(n-1)\epsilon$, since each of the $n-1$ probabilities is at least $\epsilon$, and $n-1-\E \tilde d_v=\sum_{u\ne v}(1-p_{u,v})\geq(n-1)\epsilon$, since each is at most $1-\epsilon$.

For the mixed terms, fix $(S,T)\in\Omega$, put $s=|S|$, $t=|T|$, and $Z_0=V\setminus(S\cup T)$, so that $|Z_0|=n-s-t$. By~\eqref{equ:fst},
\[
f(S,T,\E\tilde{\bm d},n)=s(n-1-t)-\sum_{u\in S}\sum_{v\ne u}p_{u,v}+\sum_{u\in T}\sum_{v\ne u}p_{u,v}.
\]
Write $s(n-1-t)=\sum_{u\in S}\sum_{v\notin T,v\ne u}1$, and split each inner sum according to whether $v$ lies in $S$, $T$, or $Z_0$. The terms $p_{u,v}$ with $u\in S$ and $v\in T$ appear once with each sign and cancel. Each pair inside $S$ or inside $T$ appears twice, once for each of its vertices. This gives
\begin{align*}
f(S,T,\E\tilde{\bm d},n)
&=2\sum_{\{u,v\}\subseteq S}(1-p_{u,v})
 +\sum_{u\in S,v\in Z_0}(1-p_{u,v})\\
&\quad+2\sum_{\{u,v\}\subseteq T}p_{u,v}
 +\sum_{u\in T,v\in Z_0}p_{u,v}.
\end{align*}
Every term is at least $\epsilon$. Counting the terms,
\[
f(S,T,\E\tilde{\bm d},n)\geq\epsilon\bigl(s(s-1)+t(t-1)+(s+t)(n-s-t)\bigr)=\epsilon\bigl((s+t)(n-1)-2st\bigr).
\]
Dividing by $s+t$ gives the first inequality in
\[
\frac{f(S,T,\E\tilde{\bm d},n)}{s+t}
\geq\epsilon\left(n-1-\frac{2st}{s+t}\right)
\geq\epsilon\left(\frac n2-1\right)
\geq\frac{\epsilon(n-1)}4.
\]
The second inequality uses $st/(s+t)\leq(s+t)/4\leq n/4$, which follows from $4st\leq(s+t)^2$. The third is equivalent to $n\geq3$. By~\eqref{equ:margin}, $\mathcal E(\E\tilde{\bm d})\geq\min\{\epsilon,\epsilon/4\}=\epsilon/4$. Since $H(\bm M)\geq r$, the condition on $r$ in Corollary~\ref{thm:eps} implies~\eqref{equ:edcond}, and Theorem~\ref{thm:mlesuff} gives the result.

\section{An interior representation for complete designs}\label{app:hyper}
\subsection{Proof of Lemma~\ref{lem:complete-rank}}
Recall that $\bm X$ has full column rank if and only if its correlation matrix $\bm M=\bm X^\top\bm X$ is positive definite, because $\bm v^\top\bm M\bm v=\norm{\bm X\bm v}^2$. Replication only repeats rows of $\bm X$ and does not change its column rank, so it suffices to consider one experiment on every $k$-set. Every vertex lies in $\binom{n-1}{k-1}$ of the $k$-sets and every pair of distinct vertices in $\binom{n-2}{k-2}$, so the correlation matrix has these numbers as its diagonal and off-diagonal entries. Since $\binom{n-1}{k-1}=\binom{n-2}{k-1}+\binom{n-2}{k-2}$, it equals $\binom{n-2}{k-1}\bm I+\binom{n-2}{k-2}\bm J$, where $\bm J$ is the all-ones matrix. For every $\bm v\ne\bm0$,
\[
\bm v^\top\left(\binom{n-2}{k-1}\bm I+\binom{n-2}{k-2}\bm J\right)\bm v
=\binom{n-2}{k-1}\norm{\bm v}^2+\binom{n-2}{k-2}(\bm1^\top\bm v)^2>0,
\]
because $\binom{n-2}{k-1}>0$ for $k\leq n-1$. Hence the correlation matrix is positive definite. By the equivalence between positive definiteness of $\bm M$ and full column rank of $\bm X$ noted at the start of the proof, the complete design therefore has full column rank.

\subsection{Proof of Proposition~\ref{pro:projection-complete}}
The proof constructs an interior representation of $\bm d(\bm y)$. It starts from the vector $\bm p$ of success probabilities, which represents the expected degrees, and adds the correction $\bm q$ of least norm that reproduces the observed degrees. If every coordinate of $\bm q$ is smaller than $\epsilon$ in absolute value, then all coordinates of $\bm p+\bm q$ lie in $(0,1)$. Most of the proof bounds the probability that some coordinate of $\bm q$ is large, using the symmetry of the complete design and Hoeffding's inequality.

\paragraph{Step 1: the construction}
Let $p_j=p_{S_j}(\bm\beta)$ be the success probability of experiment $j$, and let $\bm p=(p_1,\ldots,p_m)$. By assumption, every $p_j$ lies in $[\epsilon,1-\epsilon]$. Since $\E[Y_j]=p_j$, the vector $\bm p$ satisfies $\bm X^\top\bm p=\E[\bm d(\bm Y)]$: it represents the expected degree sequence, with weights in $[\epsilon,1-\epsilon]$. We correct it so that it represents the observed degree sequence $\bm d(\bm y)=\bm X^\top\bm y$ instead.

The complete design has full column rank (Section~\ref{sec:complete-hyper}), so $\bm M=\bm X^\top\bm X$ is invertible. Let
\[
\bm H=\bm X\bm M^{-1}\bm X^\top
\qquad\hbox{and}\qquad
\bm q=\bm H(\bm y-\bm p).
\]
The $m\times m$ matrix $\bm H$ is symmetric, and it is idempotent, that is, $\bm H^2=\bm H$, because
\[
\bm H^2=\bm X\bm M^{-1}(\bm X^\top\bm X)\bm M^{-1}\bm X^\top=\bm X\bm M^{-1}\bm X^\top=\bm H.
\]
It also satisfies $\bm X^\top\bm H=(\bm X^\top\bm X)\bm M^{-1}\bm X^\top=\bm X^\top$. Hence $\bm X^\top\bm q=\bm X^\top(\bm y-\bm p)$, and the vector $\bm w=\bm p+\bm q$ satisfies
\[
\bm X^\top\bm w=\bm X^\top\bm p+\bm X^\top(\bm y-\bm p)=\bm X^\top\bm y=\bm d(\bm y).
\]
In regression terms, $\bm H$ is the hat matrix, the orthogonal projection of $\reals^m$ onto the column space of $\bm X$, and $\bm q$ is the vector of least Euclidean norm that satisfies $\bm X^\top\bm q=\bm X^\top(\bm y-\bm p)$. If every coordinate of $\bm q$ satisfies $|q_j|<\epsilon$, then every coordinate of $\bm w$ lies in $(0,1)$, because $p_j\in[\epsilon,1-\epsilon]$. In that case, $\bm w$ is an interior representation of $\bm d(\bm y)$. It remains to bound the probability that $|q_j|\geq\epsilon$ for some $j$.

\paragraph{Step 2: concentration of each coordinate of $\bm q$}
The $j$th coordinate of $\bm q$ is
\[
q_j=\sum_lH_{j,l}(Y_l-p_l),
\]
a sum of independent random variables with mean zero. The $l$th term takes values in an interval of length $|H_{j,l}|$, because $Y_l-p_l\in\{-p_l,1-p_l\}$. Hoeffding's inequality \citep{hoeffding} therefore gives
\[
\P[|q_j|\geq\epsilon]\leq2\exp\left(-\frac{2\epsilon^2}{\sum_lH_{j,l}^2}\right).
\]
The sum in the exponent is a diagonal entry of $\bm H$. Indeed, by the definition of the matrix product, $(\bm H^2)_{j,j}=\sum_lH_{j,l}H_{l,j}$. Since $\bm H$ is symmetric, $H_{l,j}=H_{j,l}$, so $(\bm H^2)_{j,j}=\sum_lH_{j,l}^2$. Since $\bm H$ is idempotent, $(\bm H^2)_{j,j}=H_{j,j}$. Hence
\[
\sum_lH_{j,l}^2=H_{j,j}=\bm x_j^\top\bm M^{-1}\bm x_j=:h_j,
\qquad\hbox{and}\qquad
\P[|q_j|\geq\epsilon]\leq2\exp(-2\epsilon^2/h_j).
\]
The number $h_j$, the $j$th diagonal entry of the hat matrix, is called the leverage of experiment $j$ in regression. It measures how strongly the outcome of experiment $j$ influences the fitted values, and here it determines how much the correction $q_j$ fluctuates.

\paragraph{Step 3: the leverages of a complete design}
Consider first one experiment on every $k$-set. Its correlation matrix is $\bm G_{n,k}$ in~\eqref{equ:complete-Gram}, which is positive definite for $2\leq k<n$, and the leverage of the experiment on $S$ is $\ind_S^\top\bm G_{n,k}^{-1}\ind_S$. Relabelling the vertices permutes the $k$-sets and leaves $\bm G_{n,k}$ unchanged, so this leverage is the same for every $S$. The leverages sum to the trace of the hat matrix:
\[
\sum_S\ind_S^\top\bm G_{n,k}^{-1}\ind_S
=\tr\Bigl(\bm G_{n,k}^{-1}\sum_S\ind_S\ind_S^\top\Bigr)
=\tr(\bm G_{n,k}^{-1}\bm G_{n,k})=n.
\]
Hence each of the $\binom nk$ leverages equals $n/\binom nk=k/\binom{n-1}{k-1}$.

Now let every candidate edge have replication number at least $r$. Then
\[
\bm M=\sum_SM_S\ind_S\ind_S^\top\succeq r\sum_S\ind_S\ind_S^\top=r\bm G_{n,k},
\]
because $M_S\geq r$ and each matrix $\ind_S\ind_S^\top$ is positive semidefinite; here $\bm A\succeq\bm B$ means that $\bm A-\bm B$ is positive semidefinite. For positive definite matrices, $\bm A\succeq\bm B$ implies $\bm B^{-1}\succeq\bm A^{-1}$ \citep{HornJohnson2013}. Hence $\bm M^{-1}\preceq(r\bm G_{n,k})^{-1}$, and for an experiment $j$ on the candidate edge $S_j$,
\begin{equation}
h_j=\bm x_j^\top\bm M^{-1}\bm x_j
\leq\frac1r\ind_{S_j}^\top\bm G_{n,k}^{-1}\ind_{S_j}
=\frac{k}{r\binom{n-1}{k-1}}.\label{equ:leverage}
\end{equation}

\paragraph{Step 4: conclusion}
Experiments on the same candidate edge have identical rows of $\bm H$, and hence identical values of $q_j$, so there are at most $\binom nk$ distinct events $\{|q_j|\geq\epsilon\}$. By Step~2 and~\eqref{equ:leverage}, each has probability at most $2\exp(-2r\epsilon^2\binom{n-1}{k-1}/k)$, and a union bound gives the probability bound~\eqref{equ:projection-bound}.

On the complementary event, $\bm w=\bm p+\bm q\in(0,1)^m$ is an interior representation of $\bm d(\bm y)$, and a unique MLE exists by the criterion of Section~\ref{sec:model}. This can also be verified directly. Since $\bm X^\top\bm w=\bm X^\top\bm y$, the negative log-likelihood can be written with $\bm w$ in place of $\bm y$. Put $\zeta=\min_j\{w_j,1-w_j\}>0$. For each $z$, $\log(1+e^z)-w_jz\geq\zeta|z|$, since $\log(1+e^z)\geq\max\{0,z\}$. Hence
\begin{align*}
-\ell(\bm\theta)
&=\sum_j\bigl[\log(1+e^{\bm x_j^\top\bm\theta})-w_j\bm x_j^\top\bm\theta\bigr]\\
&\geq\zeta\sum_j|\bm x_j^\top\bm\theta|
\geq\zeta\norm{\bm X\bm\theta}
\geq\zeta\sqrt{\lambda_1(\bm M)}\norm{\bm\theta},
\end{align*}
using $\sum_j|z_j|\geq(\sum_jz_j^2)^{1/2}$ in the second inequality. The negative log-likelihood is therefore coercive, and it is strictly convex because $\bm X$ has full column rank, so its minimum exists and is unique. Finally, solving the probability bound~\eqref{equ:projection-bound} for $r$ gives~\eqref{equ:projection-condition}.

\subsection{Proof of Corollary~\ref{cor:hyper}}\label{app:hyper-constants}
The proof shows that the hypothesis~\eqref{equ:epscond} implies the condition~\eqref{equ:projection-condition} of Proposition~\ref{pro:projection-complete}. The only ingredient is a lower bound on $\binom{n-1}{k-1}$.

\paragraph{Step 1: a lower bound on $\binom{n-1}{k-1}$}
We show that, for $k\geq2$ and $n>4k$,
\begin{equation}
\binom{n-1}{k-1}\geq\frac{k^2}{4}\left(\frac nk\right)^{k-1}.\label{equ:D-lower}
\end{equation}
Writing out the binomial coefficient,
\[
\binom{n-1}{k-1}=\binom{n-1}{k-1}=\frac{\prod_{i=1}^{k-1}(n-i)}{(k-1)!}\geq\frac{(n-k+1)^{k-1}}{(k-1)!}\geq\frac{(3n/4)^{k-1}}{(k-1)!},
\]
where the last inequality uses $n-k+1\geq3n/4$, which holds because $n>4k$. The right-hand side equals $R_k(n/k)^{k-1}$ with $R_k=(3k/4)^{k-1}/(k-1)!$, so~\eqref{equ:D-lower} follows from $R_k\geq k^2/4$. The values $R_2=3/2\geq1$ and $R_3=81/32\geq9/4$ can be checked directly. For $k\geq3$,
\[
\frac{R_{k+1}}{R_k}
=\frac{(3(k+1)/4)^k}{k!}\cdot\frac{(k-1)!}{(3k/4)^{k-1}}
=\frac34\Bigl(1+\frac1k\Bigr)^k
=\Bigl(1+\frac1k\Bigr)^2\cdot\frac34\Bigl(1+\frac1k\Bigr)^{k-2}.
\]
Bernoulli's inequality states that $(1+x)^t\geq1+tx$ for every real $x\geq-1$ and every integer $t\geq0$. With $x=1/k$ and $t=k-2$, it gives $(1+1/k)^{k-2}\geq2-2/k$, and $\frac34(2-2/k)\geq1$ for $k\geq3$. Hence $R_{k+1}\geq(1+1/k)^2R_k$. If $R_k\geq k^2/4$, then $R_{k+1}\geq(1+1/k)^2k^2/4=(k+1)^2/4$, and induction from $k=3$ yields $R_k\geq k^2/4$ for all $k\geq2$.

\paragraph{Step 2: conclusion}
Write $A=(c+1)\log n+\log4$. The hypothesis~\eqref{equ:epscond} of Corollary~\ref{cor:hyper} states that $\epsilon\geq\sqrt{2A}\,(k/n)^{(k-1)/2}$, that is,
\[
\epsilon^2\geq2A\left(\frac kn\right)^{k-1}.
\]
Combining this with~\eqref{equ:D-lower}, and using $r\geq1$,
\[
\frac{2r\epsilon^2}{k}\binom{n-1}{k-1}
\geq\frac2k\cdot2A\left(\frac kn\right)^{k-1}\cdot\frac{k^2}{4}\left(\frac nk\right)^{k-1}
=kA.
\]
Moreover,
\[
kA=k(c+1)\log n+k\log4\geq k\log n+c\log n\geq\log\left(2\binom nk\right)+c\log n.
\]
In the first inequality, $k(c+1)\log n=k\log n+kc\log n\geq k\log n+c\log n$ because $k\geq1$, and the term $k\log4$ is nonnegative. In the second, $\binom nk\leq n^k/k!\leq n^k/2$ for $k\geq2$, so $\log(2\binom nk)\leq k\log n$. Together, $2r\epsilon^2\binom{n-1}{k-1}/k\geq\log(2\binom nk)+c\log n$, which is condition~\eqref{equ:projection-condition} rearranged. Proposition~\ref{pro:projection-complete} therefore gives failure probability at most $n^{-c}$.

\section{Separating-vector bounds}\label{app:separation}
\subsection{Proof of Lemma~\ref{lm:certificate}}
Let $\bm\alpha\ne\bm0$ be a separating vector. By Definition~\ref{def:overlap}, $\bm x_j^\top\bm\alpha\geq0$ for every experiment $j$ with $y_j=1$, and $\bm x_j^\top\bm\alpha\leq0$ for every experiment $j$ with $y_j=0$. Let $P=\{v:\alpha_v>0\}$, $N=\{v:\alpha_v<0\}$, and $Z_0=V\setminus(P\cup N)$, so that $\alpha_v=0$ for $v\in Z_0$.

Consider an experiment $j$ in the first sum of~\eqref{equ:W}, so that its group $S_j$ lies in $P\cup Z_0$ and meets $P$. The coordinates of $\bm\alpha$ on $S_j$ are positive on $S_j\cap P\ne\emptyset$ and zero on $S_j\cap Z_0$, so
\[
\bm x_j^\top\bm\alpha=\sum_{v\in S_j}\alpha_v=\sum_{v\in S_j\cap P}\alpha_v>0.
\]
If $y_j$ were $0$, the separating vector would need $\bm x_j^\top\bm\alpha\leq0$, a contradiction. Hence $y_j=1$. In the same way, an experiment $j$ in the second sum has $S_j\subseteq N\cup Z_0$ and $S_j\cap N\ne\emptyset$, so $\bm x_j^\top\bm\alpha=\sum_{v\in S_j\cap N}\alpha_v<0$, which forces $y_j=0$.

\subsection{Proof of Proposition~\ref{pro:master}}
The proof shows that the event that no unique finite MLE exists is contained in a union of simple events, one for each sign pattern $(P,N)$: the event that all experiments counted by $W(P,N)$ have prescribed outcomes. Each such event requires $W(P,N)$ independent outcomes to take given values, so its probability is at most $(1-\epsilon)^{W(P,N)}$, and a union bound gives~\eqref{equ:master}.

Consider first a fixed design, and let $A$ be the event that no unique finite MLE exists. For disjoint $P,N\subseteq V$, not both empty, put $Z_0=V\setminus(P\cup N)$, and let $J_1(P,N)$ be the set of experiments $j$ with $S_j\subseteq P\cup Z_0$ and $S_j\cap P\ne\emptyset$, and $J_0(P,N)$ the set of experiments $j$ with $S_j\subseteq N\cup Z_0$ and $S_j\cap N\ne\emptyset$. Let $A_{P,N}$ be the event that $Y_j=1$ for every $j\in J_1(P,N)$ and $Y_j=0$ for every $j\in J_0(P,N)$.

We first show that $A\subseteq\bigcup_{(P,N)}A_{P,N}$, where the union is over all disjoint $P,N\subseteq V$ that are not both empty. This inclusion means that every outcome vector $\bm y$ for which no unique finite MLE exists belongs to at least one of the events $A_{P,N}$. The pattern $(P,N)$ depends on $\bm y$, which is why the union runs over all patterns.

Fix an outcome vector $\bm y$ in $A$. By Proposition~\ref{thm:mle-overlap}, no unique finite MLE exists exactly when $(\bm X,\bm y)$ violates the overlapping condition, that is, when a separating vector $\bm\alpha\ne\bm0$ exists. By Definition~\ref{def:overlap}, such a vector satisfies $\bm x_j^\top\bm\alpha\geq0$ for every $j$ with $y_j=1$ and $\bm x_j^\top\bm\alpha\leq0$ for every $j$ with $y_j=0$. Put $P=\{v:\alpha_v>0\}$ and $N=\{v:\alpha_v<0\}$. These sets are disjoint, and they are not both empty because $\bm\alpha\ne\bm0$; the coordinates of $\bm\alpha$ on $Z_0=V\setminus(P\cup N)$ are zero.

By Lemma~\ref{lm:certificate}, every experiment in $J_1(P,N)$ has outcome $1$ and every experiment in $J_0(P,N)$ has outcome $0$: each experiment in $J_1(P,N)$ has $\bm x_j^\top\bm\alpha>0$, which a separating vector allows only if $y_j=1$, and each experiment in $J_0(P,N)$ has $\bm x_j^\top\bm\alpha<0$, which it allows only if $y_j=0$. Thus $\bm y$ belongs to $A_{P,N}$ for this particular pattern $(P,N)$, and since $\bm y\in A$ was arbitrary, $A\subseteq\bigcup_{(P,N)}A_{P,N}$.

Next, the sets $J_1(P,N)$ and $J_0(P,N)$ are disjoint: an experiment in $J_1(P,N)$ has $S_j\subseteq P\cup Z_0$, so $S_j\cap N=\emptyset$, and hence it is not in $J_0(P,N)$. Their total size is
\[
|J_1(P,N)|+|J_0(P,N)|=\sum_{S\in E:\,S\subseteq P\cup Z_0,\ S\cap P\ne\emptyset}M_S+\sum_{S\in E:\,S\subseteq N\cup Z_0,\ S\cap N\ne\emptyset}M_S=W(P,N),
\]
by~\eqref{equ:W}. The outcomes are independent, $\P[Y_j=1]=p_{S_j}(\bm\beta)\leq1-\epsilon$, and $\P[Y_j=0]=1-p_{S_j}(\bm\beta)\leq1-\epsilon$. Hence
\[
\P[A_{P,N}]
=\prod_{j\in J_1(P,N)}p_{S_j}(\bm\beta)\prod_{j\in J_0(P,N)}\bigl(1-p_{S_j}(\bm\beta)\bigr)
\leq(1-\epsilon)^{W(P,N)}.
\]
A union bound over the pairs $(P,N)$ gives
\[
\P[A]\leq\sum_{(P,N)}\P[A_{P,N}]\leq\sum_{(P,N)}(1-\epsilon)^{W(P,N)},
\]
which is~\eqref{equ:master}.

Now let the design be random, and suppose that, conditional on the design $S_1,\ldots,S_m$, the outcomes are independent with $\P[Y_j=1\mid S_1,\ldots,S_m]=p_{S_j}(\bm\beta)$. The argument above then applies conditionally on the design, with $W(P,N)$ computed from the sampled design, and gives
\[
\P[A\mid S_1,\ldots,S_m]\leq\sum_{(P,N)}(1-\epsilon)^{W(P,N)}.
\]
Taking expectations over the design proves the second statement.

\subsection{Proof of Theorem~\ref{thm:hyper-sharp}}
By Proposition~\ref{pro:master},
\[
\P[\text{no unique finite MLE}]\leq\sum_{(P,N)}(1-\epsilon)^{W(P,N)},
\]
where the sum runs over all disjoint $P,N\subseteq V$ that are not both empty. The proof bounds this sum in three steps. First, we show that in a complete design $W(P,N)$ grows at least linearly in the size $s=|P|+|N|$ of the support of the pattern. Second, we count the patterns with support of size $s$; their number grows only exponentially in $s$. Third, we group the terms of the sum by $s$: under condition~\eqref{equ:hyper-sharp}, the decay of each term outweighs the growth of their number, and the sum is bounded by a convergent geometric series.

\paragraph{Step 1: a lower bound on $W(P,N)$}
Put $B_{n,k}=\binom{\lfloor n/2\rfloor-1}{k-1}$. We show that, for every pattern $(P,N)$ with $|P|+|N|=s$,
\begin{equation}
W(P,N)\geq\frac{rs}kB_{n,k}.\label{equ:W-lower}
\end{equation}
Recall from~\eqref{equ:W} that, with $Z_0=V\setminus(P\cup N)$,
\[
W(P,N)=\sum_{S\in E:\,S\subseteq P\cup Z_0,\ S\cap P\ne\emptyset}M_S
+\sum_{S\in E:\,S\subseteq N\cup Z_0,\ S\cap N\ne\emptyset}M_S,
\]
the number of experiments on candidate edges that meet one of the two sets and avoid the other. In a complete $k$-uniform design, $E$ consists of all $k$-subsets of $V$, and $M_S\geq r$ for every $S\in E$.

Write $p=|P|$ and $q=|N|$, so that $s=p+q$. The first sum runs over the $k$-sets that meet $P$ and avoid $N$. A vertex $v\in P$ lies in $\binom{n-q-1}{k-1}$ of them, since the other $k-1$ vertices are chosen from the $n-q-1$ vertices outside $N\cup\{v\}$. Counting the pairs $(v,S)$ with $v\in S\cap P$ therefore gives $p\binom{n-q-1}{k-1}$ pairs, and each $S$ appears in at most $k$ of them. Hence there are at least $(p/k)\binom{n-q-1}{k-1}$ such $k$-sets, each tested at least $r$ times. The same argument applies to the second sum with the roles of $P$ and $N$ exchanged, and so
\begin{equation}
W(P,N)\geq\frac rk\left[
p\binom{n-q-1}{k-1}+q\binom{n-p-1}{k-1}\right].\label{equ:count-incidences}
\end{equation}
It remains to bound the right-hand side of~\eqref{equ:count-incidences} from below in terms of $s$ alone. Put $F(j)=\binom j{k-1}$, with $F(j)=0$ for $j<k-1$. Then $F$ has nondecreasing increments: $F(j+1)-F(j)=\binom j{k-2}$, which is nondecreasing in $j$. The bound~\eqref{equ:count-incidences} is symmetric in $(P,N)$, so we may assume $p\geq q$. Suppose first that $q>0$, and put $j_1=n-q-1$ and $j_2=n-p-1$, so that $j_1\geq j_2$ and the bracket in~\eqref{equ:count-incidences} equals $pF(j_1)+qF(j_2)$. We show
\begin{equation}
pF(j_1)+qF(j_2)
\geq\frac s2\bigl(F(j_1)+F(j_2)\bigr)
\geq sF(j_0)
\geq sB_{n,k},
\qquad j_0=\left\lfloor\frac{j_1+j_2}{2}\right\rfloor.\label{equ:three-steps}
\end{equation}
The first inequality holds because the difference of the two sides equals $\frac{p-q}2(F(j_1)-F(j_2))$, which is nonnegative since $p\geq q$ and $F$ is nondecreasing. For the second, $j_1-j_0\geq j_0-j_2$, so $F(j_1)-F(j_0)$ is a sum of at least as many increments as $F(j_0)-F(j_2)$, and each of these increments is at least as large as each increment in $F(j_0)-F(j_2)$, because the increments are nondecreasing. Hence $F(j_1)-F(j_0)\geq F(j_0)-F(j_2)$, that is, $F(j_1)+F(j_2)\geq2F(j_0)$. For the third, $j_0=\lfloor n-1-s/2\rfloor\geq\lfloor n/2\rfloor-1$, since $s\leq n$, and $F$ is nondecreasing. If $q=0$, then~\eqref{equ:count-incidences} gives $W(P,N)\geq(rp/k)F(n-1)$ directly, and $F(n-1)\geq B_{n,k}$. In both cases,~\eqref{equ:W-lower} follows.

\paragraph{Step 2: the number of patterns with support of size $s$}
A pattern $(P,N)$ with $|P|+|N|=s$ is determined by its support $P\cup N$, which can be chosen in $\binom ns$ ways, and by the assignment of each of its $s$ vertices to $P$ or to $N$, which can be made in $2^s$ ways. Since $\binom ns\leq n^s$, there are at most $(2n)^s$ such patterns.

\paragraph{Step 3: bounding the sum}
By~\eqref{equ:W-lower} and $1-\epsilon\leq e^{-\epsilon}$, each pattern with support of size $s$ contributes at most
\[
(1-\epsilon)^{W(P,N)}\leq\exp\left(-\epsilon\frac{rs}kB_{n,k}\right).
\]
Grouping the patterns by $s$ and using Step~2,
\[
\sum_{(P,N)}(1-\epsilon)^{W(P,N)}
\leq\sum_{s=1}^n(2n)^s\exp\left(-\epsilon\frac{rs}kB_{n,k}\right)
=\sum_{s=1}^ne^{-s\theta},
\qquad
\theta=\frac{\epsilon rB_{n,k}}k-\log(2n).
\]
Condition~\eqref{equ:hyper-sharp} states that $\epsilon rB_{n,k}/k\geq(1+a)\log n+3$. Hence
\[
\theta\geq(1+a)\log n+3-\log2-\log n=a\log n+3-\log2\geq a\log n+2,
\]
since $\log2<1$. In particular $\theta>0$, so the terms $e^{-s\theta}$ form a geometric series with ratio $e^{-\theta}<1$, and
\[
\sum_{s=1}^ne^{-s\theta}\leq\sum_{s=1}^\infty e^{-s\theta}=\frac{e^{-\theta}}{1-e^{-\theta}}.
\]
Finally, $e^{-\theta}\leq e^{-2}n^{-a}$ and $1-e^{-\theta}\geq1-e^{-2}$, so
\[
\P[\text{no unique finite MLE}]\leq\frac{e^{-2}n^{-a}}{1-e^{-2}}=\frac{n^{-a}}{e^2-1}<n^{-a}.
\]

\subsection{Proof of Corollary~\ref{cor:graph-sharp}}
The proof follows the three steps of the proof of Theorem~\ref{thm:hyper-sharp}, with two changes. For graphs, $W(P,N)$ can be counted exactly, which gives a larger lower bound, and the patterns are counted more precisely. A fourth step checks that the condition of the corollary suffices.

\paragraph{Step 1: a lower bound on $W(P,N)$}
For a complete graph design, the candidate edges are all pairs of vertices, and each is tested at least $r$ times. Let $(P,N)$ be a pattern with $p=|P|$, $q=|N|$, $s=p+q$, and $Z_0=V\setminus(P\cup N)$, so that $|Z_0|=n-s$. The first family in~\eqref{equ:W} consists of the pairs inside $P\cup Z_0$ that meet $P$: the $\binom p2$ pairs inside $P$ and the $p(n-s)$ pairs joining $P$ and $Z_0$. Similarly, the second family consists of $\binom q2$ pairs inside $N$ and $q(n-s)$ pairs joining $N$ and $Z_0$. Hence
\[
W(P,N)\geq r\left[\binom p2+\binom q2+s(n-s)\right].
\]
Since $p^2+q^2\geq s^2/2$, we have $\binom p2+\binom q2=(p^2+q^2-s)/2\geq s^2/4-s/2$, and therefore
\begin{equation}
W(P,N)\geq r\,\omega(s),\qquad
\omega(s)=\frac{s^2}4-\frac s2+s(n-s)=s\left(n-\frac{3s}{4}-\frac12\right).\label{equ:W-graph}
\end{equation}

\paragraph{Step 2: the number of patterns with support of size $s$}
As in the proof of Theorem~\ref{thm:hyper-sharp}, there are at most $\binom ns2^s$ such patterns. Using $\binom ns\leq(en/s)^s$, their number is at most $(2en/s)^s$.

\paragraph{Step 3: a sufficient condition for each support size}
By Proposition~\ref{pro:master}, \eqref{equ:W-graph}, Step~2, and $1-\epsilon\leq e^{-\epsilon}$,
\[
\P[\text{no unique finite MLE}]
\leq\sum_{s=1}^n\left(\frac{2en}s\right)^se^{-\epsilon r\omega(s)}.
\]
Suppose that, for every $s=1,\ldots,n$,
\begin{equation}
\epsilon r\omega(s)\geq s\log\frac{2en}s+a\log n+s.\label{equ:graph-per-s}
\end{equation}
Then the $s$th term is at most $n^{-a}e^{-s}$, and the probability is at most $n^{-a}\sum_{s\geq1}e^{-s}=n^{-a}/(e-1)<n^{-a}$. It remains to show that the condition of the corollary implies~\eqref{equ:graph-per-s}.

\paragraph{Step 4: the condition of the corollary implies~\eqref{equ:graph-per-s}}
The condition of the corollary states that $\epsilon r\geq2((1+a)\log n+3)/(n-1)$. Hence
\[
\frac{\epsilon r\omega(s)}s\geq2\bigl((1+a)\log n+3\bigr)\frac{n-3s/4-1/2}{n-1},
\]
and, after division by $s$,~\eqref{equ:graph-per-s} holds whenever $g_a(s)\geq0$, where
\[
g_a(s)=2\bigl((1+a)\log n+3\bigr)\frac{n-3s/4-1/2}{n-1}-\log\frac{2en}s-1-\frac{a\log n}{s}.
\]
We show that $g_a(s)\geq0$ for all real $s\in[1,n]$. As a function of $s$, the first term is linear, $-\log(2en/s)=\log s-\log(2en)$ is concave, and $-a\log n/s$ is concave; indeed,
\[
g_a''(s)=-\frac1{s^2}-\frac{2a\log n}{s^3}<0.
\]
A concave function on an interval attains its minimum at an endpoint, so it suffices to check $s=1$ and $s=n$.

At $s=1$, the fraction equals $(n-5/4)/(n-1)\geq3/4$, and $\log(2en)=\log n+\log2+1$, so
\[
g_a(1)\geq\frac32\bigl((1+a)\log n+3\bigr)-\log n-\log2-2-a\log n
=\frac{1+a}2\log n+\frac52-\log2>0.
\]
At $s=n$, the fraction equals $(n-2)/(4(n-1))$, and $\log(2en/n)=\log2+1$, so
\[
g_a(n)=\frac{n-2}{2(n-1)}\bigl((1+a)\log n+3\bigr)-\log2-2-\frac{a\log n}n.
\]
For $n\geq40$, we have $(n-2)/(2(n-1))\geq19/39$ and $\log n\geq\log32=5\log2$. The part of $g_a(n)$ that does not involve $a$ is therefore at least
\[
\frac{19}{39}(5\log2+3)-\log2-2=\frac{56}{39}\log2-\frac7{13}\geq\frac{28}{39}-\frac{21}{39}=\frac7{39}>0,
\]
using $\log2\geq1/2$. The coefficient of $a$ is $\bigl(\frac{n-2}{2(n-1)}-\frac1n\bigr)\log n>0$. Hence $g_a(n)>0$, so $g_a\geq0$ on $[1,n]$, which proves~\eqref{equ:graph-per-s} and the corollary. For balanced designs, $m=r\binom n2$, and the condition on $r$ becomes $m\geq\epsilon^{-1}n((1+a)\log n+3)$.

\section{Estimation bounds}\label{app:mse1}
\subsection{Proof of Theorem~\ref{lem:mse1}}
The proof compares the estimate with the true parameter through the curvature of the likelihood. On $\Theta_b$, the negative log-likelihood is strongly convex in the norm induced by $\bm M$, with modulus $q_{kb}$, so the estimation error is controlled by the size of the score vector $\bm g$ at the true parameter. The expected squared size of $\bm g$, measured in the norm induced by $\bm M^{-1}$, is at most $n/4$, which gives~\eqref{equ:expected-error}, and a bounded differences inequality gives~\eqref{equ:l2bound}.

Throughout, put $f=-\ell$, $p_j=p_{S_j}(\bm\beta)$, and $\norm{\bm v}_{\bm M}^2=\bm v^\top\bm M\bm v$. By~\eqref{equ:hessian}, the Hessian of $f$ at any $\bm\beta'$ is $\nabla^2f(\bm\beta')=\sum_j\sigma'(\bm x_j^\top\bm\beta')\bm x_j\bm x_j^\top$.

\paragraph{Step 1: the estimate exists and is unique}
The function $f$ is continuous, and $\Theta_b$ is compact, so $f$ attains its minimum on $\Theta_b$. Since $\sigma'>0$ and $\lambda_1(\bm M)>0$, the Hessian $\nabla^2f(\bm\beta')$ is positive definite at every $\bm\beta'$, so $f$ is strictly convex, and its minimizer over the convex set $\Theta_b$ is unique.

\paragraph{Step 2: a lower bound on the curvature on $\Theta_b$}
For $\bm\beta'\in\Theta_b$ and every experiment $j$, $|\bm x_j^\top\bm\beta'|\leq\norm{\bm x_j}_1\norm{\bm\beta'}_\infty\leq kb$. The function $\sigma'(z)=e^z/(1+e^z)^2$ is even and decreasing in $|z|$, so $\sigma'(\bm x_j^\top\bm\beta')\geq\sigma'(kb)=q_{kb}$. Hence
\[
\nabla^2f(\bm\beta')\succeq q_{kb}\sum_j\bm x_j\bm x_j^\top=q_{kb}\bm M
\qquad\hbox{for every }\bm\beta'\in\Theta_b.
\]

\paragraph{Step 3: strong convexity along the segment from $\bm\beta$ to $\hat{\bm\beta}_b$}
Put $\bm\Delta=\hat{\bm\beta}_b-\bm\beta$. Both the true parameter $\bm\beta$ and the estimate $\hat{\bm\beta}_b$ lie in $\Theta_b$, and $\Theta_b$ is convex, so every point $\bm\beta+t\bm\Delta$ with $t\in[0,1]$, on the line segment between them, also lies in $\Theta_b$. By the fundamental theorem of calculus, applied to $t\mapsto\nabla f(\bm\beta+t\bm\Delta)$,
\[
\bigl(\nabla f(\hat{\bm\beta}_b)-\nabla f(\bm\beta)\bigr)^\top\bm\Delta
=\int_0^1\bm\Delta^\top\nabla^2f(\bm\beta+t\bm\Delta)\bm\Delta\,dt
\geq q_{kb}\norm{\bm\Delta}_{\bm M}^2,
\]
where the inequality uses Step~2 at each point of the segment.

\paragraph{Step 4: the optimality condition and the score}
Since $\hat{\bm\beta}_b$ minimizes the convex function $f$ over the convex set $\Theta_b$, the first-order optimality condition states that $\nabla f(\hat{\bm\beta}_b)^\top(\bm\beta'-\hat{\bm\beta}_b)\geq0$ for every $\bm\beta'\in\Theta_b$. Taking $\bm\beta'=\bm\beta$ gives $\nabla f(\hat{\bm\beta}_b)^\top\bm\Delta\leq0$. At the true parameter, $\nabla f(\bm\beta)=-\bm g$, where
\[
\bm g=\nabla\ell(\bm\beta)=\sum_j(Y_j-p_j)\bm x_j
\]
is the score vector. Substituting both facts into Step~3,
\[
q_{kb}\norm{\bm\Delta}_{\bm M}^2
\leq\nabla f(\hat{\bm\beta}_b)^\top\bm\Delta+\bm g^\top\bm\Delta
\leq\bm g^\top\bm\Delta.
\]

\paragraph{Step 5: from the score to the error}
Writing $\bm g^\top\bm\Delta=(\bm M^{-1/2}\bm g)^\top(\bm M^{1/2}\bm\Delta)$, the Cauchy--Schwarz inequality gives $\bm g^\top\bm\Delta\leq\norm{\bm M^{-1/2}\bm g}\norm{\bm\Delta}_{\bm M}$. Combined with Step~4, this gives $\norm{\bm\Delta}_{\bm M}\leq\norm{\bm M^{-1/2}\bm g}/q_{kb}$. Finally, $\norm{\bm\Delta}_{\bm M}^2\geq\lambda_1(\bm M)\norm{\bm\Delta}^2$, so
\begin{equation}
\norm{\bm\Delta}^2\leq
\frac{\norm{\bm M^{-1/2}\bm g}^2}{q_{kb}^2\lambda_1(\bm M)}.\label{equ:score-comparison}
\end{equation}

\paragraph{Step 6: the expected error}
The outcomes are independent with $\E[Y_j]=p_j$ and $\operatorname{Var}(Y_j)=p_j(1-p_j)\leq1/4$. Hence $\E\bm g=\bm0$ and the covariance matrix of $\bm g$ is $\sum_jp_j(1-p_j)\bm x_j\bm x_j^\top\preceq\frac14\bm M$. Therefore
\[
\E\norm{\bm M^{-1/2}\bm g}^2=\E[\bm g^\top\bm M^{-1}\bm g]
=\tr\left(\bm M^{-1}\sum_jp_j(1-p_j)\bm x_j\bm x_j^\top\right)
\leq\frac14\tr(\bm M^{-1}\bm M)=\frac n4.
\]
Taking expectations in~\eqref{equ:score-comparison} and dividing by $n$ proves~\eqref{equ:expected-error}.

\paragraph{Step 7: the high-probability bound}
Let $F(\bm Y)=\norm{\bm M^{-1/2}\bm g}$, a function of the independent outcomes $Y_1,\ldots,Y_m$. Changing a single outcome $Y_j$ changes $\bm g$ by $\pm\bm x_j$, and hence, by the triangle inequality, changes $F$ by at most $\norm{\bm M^{-1/2}\bm x_j}=\sqrt{h_j}$, where $h_j=\bm x_j^\top\bm M^{-1}\bm x_j$. These bounds satisfy
\[
\sum_jh_j=\tr\Bigl(\bm M^{-1}\sum_j\bm x_j\bm x_j^\top\Bigr)=\tr(\bm M^{-1}\bm M)=n.
\]
By Jensen's inequality and Step~6, $\E F\leq(\E F^2)^{1/2}\leq\sqrt n/2$. The bounded differences inequality \citep{McDiarmid1989} states that $\P[F>\E F+t]\leq\exp(-2t^2/\sum_jh_j)$ for every $t>0$, so
\[
\P\left[F>\frac{\sqrt n}2+t\right]\leq e^{-2t^2/n}.
\]
With $t=\sqrt{n\log(1/\eta)/2}$, the right-hand side equals $\eta$. Hence, with probability at least $1-\eta$,
\[
F\leq\sqrt n\left(\frac12+\sqrt{\frac{\log(1/\eta)}2}\right),
\]
and by~\eqref{equ:score-comparison}, $\norm{\bm\Delta}\leq F/(q_{kb}\sqrt{\lambda_1(\bm M)})$. Dividing by $\sqrt n$ proves~\eqref{equ:l2bound}.

\paragraph{Step 8: the unconstrained MLE}
On the event $B_b$, the unconstrained MLE $\hat{\bm\beta}$ exists and lies in $\Theta_b$. It maximizes $\ell$ over all of $\reals^n$, so in particular over $\Theta_b$, and by the uniqueness in Step~1 it equals $\hat{\bm\beta}_b$. The bound of Step~7 therefore applies to $\hat{\bm\beta}$ on $B_b$, which proves the final statement of the theorem.

\subsection{Proof of Proposition~\ref{pro:minimax}}
The proof is a hypercube argument. It considers the $2^n$ parameters with coordinates $\pm\rho$ and shows that, for two of them that differ in a single coordinate, the distributions of the outcomes are too close to be told apart reliably. Hence every estimator gets the sign of each coordinate wrong with probability at least $1/4$ on average, and each wrong sign costs a squared error of at least $\rho^2$.

\paragraph{Step 1: the hypercube}
Put $\rho^2=\min\{b^2,1/d_{\max}\}$, with $\rho=b$ if $d_{\max}=0$, and for $\bm\tau\in\{-1,1\}^n$ put $\bm\beta_{\bm\tau}=\rho\bm\tau$. Since $\rho\leq b$, all $2^n$ points $\bm\beta_{\bm\tau}$ lie in $\Theta_b$. Write $P_{\bm\tau}$ and $\E_{\bm\tau}$ for the distribution of the outcomes and the expectation under $\bm\beta_{\bm\tau}$. The supremum over $\Theta_b$ is at least the average over the hypercube, so for every estimator $\tilde{\bm\beta}$,
\begin{equation}
\sup_{\bm\beta\in\Theta_b}\E_{\bm\beta}\norm{\tilde{\bm\beta}-\bm\beta}^2
\geq\frac1{2^n}\sum_{\bm\tau}\E_{\bm\tau}\norm{\tilde{\bm\beta}-\rho\bm\tau}^2
=\sum_{v=1}^n\frac1{2^n}\sum_{\bm\tau}\E_{\bm\tau}(\tilde\beta_v-\rho\tau_v)^2.\label{equ:hypercube-average}
\end{equation}

\paragraph{Step 2: from squared errors to sign errors}
Fix a vertex $v$, and let $\hat\tau_v=1$ if $\tilde\beta_v\geq0$ and $\hat\tau_v=-1$ otherwise. If $\hat\tau_v\ne\tau_v$, then $\tilde\beta_v$ and $\rho\tau_v$ lie on opposite sides of zero, or $\tilde\beta_v=0$, so $|\tilde\beta_v-\rho\tau_v|\geq\rho$. Hence
\[
\E_{\bm\tau}(\tilde\beta_v-\rho\tau_v)^2\geq\rho^2P_{\bm\tau}[\hat\tau_v\ne\tau_v].
\]

\paragraph{Step 3: pairing neighboring corners}
For $\bm\tau\in\{-1,1\}^n$, let $\bm\tau^{(v)}$ be the corner obtained by flipping the sign of coordinate $v$. The map $\bm\tau\mapsto\bm\tau^{(v)}$ pairs the corners, so
\[
\frac1{2^n}\sum_{\bm\tau}P_{\bm\tau}[\hat\tau_v\ne\tau_v]
=\frac1{2^n}\sum_{\bm\tau}\frac12\Bigl(P_{\bm\tau}[\hat\tau_v\ne\tau_v]+P_{\bm\tau^{(v)}}[\hat\tau_v\ne\tau^{(v)}_v]\Bigr).
\]
Since $\tau^{(v)}_v=-\tau_v$ and $\hat\tau_v\in\{-1,1\}$, the event $\{\hat\tau_v\ne\tau^{(v)}_v\}$ is the complement of the event $E=\{\hat\tau_v\ne\tau_v\}$. For any two distributions $P$ and $Q$ and any event $E$, $Q[E^c]=1-Q[E]\geq1-P[E]-\operatorname{TV}(P,Q)$, where $\operatorname{TV}$ is the total variation distance. Hence each bracket is at least $1-\operatorname{TV}(P_{\bm\tau},P_{\bm\tau^{(v)}})$.

\paragraph{Step 4: neighboring corners are close}
Under $\bm\beta$, the outcomes are independent Bernoulli variables with natural parameters $z_j=\bm x_j^\top\bm\beta$ and log-partition function $A(z)=\log(1+e^z)$. For exponential families, the Kullback--Leibler divergence is the Bregman divergence of $A$, so for parameters $\bm\beta,\bm\beta'$ with natural parameters $z_j,z_j'$,
\[
\KL(P_{\bm\beta}\Vert P_{\bm\beta'})=\sum_j\bigl[A(z_j')-A(z_j)-A'(z_j)(z_j'-z_j)\bigr]
\leq\frac12\sup_zA''(z)\sum_j(z_j'-z_j)^2
\leq\frac18\norm{\bm X(\bm\beta-\bm\beta')}^2,
\]
by Taylor's theorem and $A''=\sigma'\leq1/4$. For the neighboring corners $\bm\tau$ and $\bm\tau^{(v)}$, $\bm\beta_{\bm\tau}-\bm\beta_{\bm\tau^{(v)}}=\pm2\rho\bm e_v$, and $\norm{\bm X\bm e_v}^2=M_{v,v}$, so
\[
\KL(P_{\bm\tau}\Vert P_{\bm\tau^{(v)}})\leq\frac18\cdot4\rho^2M_{v,v}=\frac{\rho^2M_{v,v}}2\leq\frac{\rho^2d_{\max}}2\leq\frac12.
\]
Pinsker's inequality, $\operatorname{TV}(P,Q)\leq\sqrt{\KL(P\Vert Q)/2}$ \citep{Tsybakov2009}, then gives $\operatorname{TV}(P_{\bm\tau},P_{\bm\tau^{(v)}})\leq1/2$. If $d_{\max}=0$, the outcomes do not depend on $\bm\beta$, and the total variation distance is zero.

\paragraph{Step 5: conclusion}
By Steps~3 and~4, each bracket in Step~3 is at least $1/2$, so the average sign-error probability for coordinate $v$ is at least $1/4$. By Step~2, the average squared error in coordinate $v$ is at least $\rho^2/4$. Summing over the $n$ coordinates in~\eqref{equ:hypercube-average} gives
\[
\sup_{\bm\beta\in\Theta_b}\frac1n\E_{\bm\beta}\norm{\tilde{\bm\beta}-\bm\beta}^2\geq\frac{\rho^2}4=\frac14\min\{b^2,1/d_{\max}\}.
\]
This holds for every estimator $\tilde{\bm\beta}$, which proves~\eqref{equ:minimax}.

\section{Rank of random designs}
\subsection{A quantitative square-matrix bound}\label{app:quantitative}
\input{quantitative-rank}

\subsection{Proof of Theorem~\ref{thm:rank-growing}}\label{app:rank-growing}
\input{rank-growing}

\subsection{Proof of Theorem~\ref{thm:rank-fixed}}\label{app:rank-fixed}
\input{rank-fixed}

\section{Random-design existence and eigenvalues}\label{app:suffrnd}
\subsection{Proof of Theorem~\ref{thm:suffrnd}}
Fix disjoint $P,N$ with total support $s$, and let $\pi(P,N)$ be the probability that one uniformly sampled $k$-set belongs to one of the two families in~\eqref{equ:W}. Applied to one experiment on every $k$-set, that is, with $r=1$, the bound~\eqref{equ:W-lower} shows that at least $sB_{n,k}/k$ of the $\binom nk$ $k$-sets belong to these families, so
\[
\pi(P,N)\geq\frac{sB_{n,k}}{k\binom nk}.
\]
The number $W(P,N)$ of sampled experiments in these families has distribution $\operatorname{Bin}(m,\pi(P,N))$. Consequently
\[
\E[(1-\epsilon)^{W(P,N)}]
=(1-\epsilon\pi(P,N))^m
\leq e^{-m\epsilon\pi(P,N)}.
\]
Summing over at most $(2n)^s$ pairs gives the same geometric series as in the proof of Theorem~\ref{thm:hyper-sharp}, with $r$ replaced by $m/\binom nk$. This yields~\eqref{equ:suffrnd}. For graphs, the exact bound is $\pi(P,N)\geq\omega(s)/\binom n2$. Set $\epsilon r$ in the graph proof equal to $\epsilon m/\binom n2$. Its sufficient condition becomes~\eqref{equ:suffrnd-graph}.

\subsection{Proof of Proposition~\ref{lem:lambdaconc}}
The matrix $\bm M$ is the sum of independent positive semidefinite matrices $\bm x_j\bm x_j^\top$, each of spectral norm $k$. The matrix Chernoff inequality \citep[Theorem 5.1.1]{T15} gives
\[
\P[\lambda_1(\bm M)\leq(1-\delta)\lambda_1(\E\bm M)]
\leq n\exp\left(-\frac{\delta^2\lambda_1(\E\bm M)}{2k}\right).
\]
Substitution of~\eqref{equ:lam1em} proves the result.

%% file: quantitative-rank.tex
The proof of Theorem~\ref{thm:rank-growing} in Appendix~\ref{app:rank-growing} reduces the rank of the random design matrix to the singularity of square matrices with independent uniformly random rows of a fixed weight. For this model, \cite{FKS22} proved that the singularity probability tends to zero when the weight is at least $(1+\varepsilon)\log n$ and at most $n/2$. Our application needs a statement of this result with two specific properties. First, we need an explicit bound on the singularity probability, which determines the rate in Theorem~\ref{thm:rank-growing}. Second, we need this bound to hold uniformly in the weight, because the weight in our construction depends on $n$ and $k$ and varies between $(1+\varepsilon)\log n$ and $2(1+\varepsilon)\log n$. Lemma~\ref{lm:quantitative-rank} is a quantitative version of the result of \cite{FKS22} with these two properties. Its proof follows their argument, uses two of their lemmas as stated, and keeps track of the constants in the remaining steps.

\begin{lm}[Quantitative square-matrix bound]\label{lm:quantitative-rank}
Let $n$ and $d$ be integers with $1\leq d\leq n$, and let $\bm Q$ be the random $n\times n$ matrix whose rows $\bm R_1,\ldots,\bm R_n$ are independent and each uniformly distributed on
\[
\{\bm x\in\{0,1\}^n:\norm{\bm x}_1=d\}.
\]
Equivalently, each row is the incidence vector of a $d$-subset of $\{1,\ldots,n\}$ chosen uniformly at random, independently of the other rows; every entry of $\bm Q$ is $0$ or $1$, and every row has exactly $d$ entries equal to $1$. We call $d$ the weight of the rows. Fix $\varepsilon\in(0,1)$. There are constants $C_\varepsilon>0$ and $\gamma_\varepsilon>0$ and an integer $n_\varepsilon$, depending only on $\varepsilon$, such that, for all $n\geq n_\varepsilon$ and all integers $d$ with $(1+\varepsilon)\log n\leq d\leq n/2$,
\[
\P[\bm Q\text{ is singular}]
\leq C_\varepsilon\left(n^{-\gamma_\varepsilon}\log d+\frac{(\log d)^{3/2}}{\sqrt d}\right),
\]
where $\bm Q$ is singular if $\det\bm Q=0$.
\end{lm}

\begin{proof}
A singular matrix $\bm Q$ has a nonzero vector in its left or its right kernel. The proof excludes two kinds of kernel vectors by union bounds: left-kernel vectors of small support, and right-kernel vectors with one large set of equal coordinates. For such vectors, after reduction modulo a suitable integer, a random row is unlikely to be orthogonal to the vector. For all remaining vectors, an anti-concentration bound of \citet{FKS22} applies, and it gives the dominant term $(\log d)^{3/2}/\sqrt d$. The work needed beyond their argument is to keep all constants uniform in $d$.

We follow the argument of \citet{FKS22}, keeping the bounds uniform in $d$,
and use their Lemmas~2.1 and~4.2. Let $\bm R_1,\ldots,\bm R_n$ be the rows of $\bm Q$.
Write $p=d/n$, choose
$\delta=\varepsilon/[64(1+\varepsilon)]$, and later choose a sufficiently
small constant $a=a_\varepsilon>0$. Put $t=an/\log d$.

We first record uniform versions of the two scalar estimates used in their proof.
For a Bernoulli sum of $s$ terms, its probability of being even is
\[
 B_{s,p}=\frac{1+(1-2p)^s}{2}
 \leq\frac{1+e^{-2sp}}2
 =e^{-sp}\cosh(sp).
\]
Writing $u=sp$, this gives $B_{s,p}\leq\exp\{-(1-\delta/2)u\}$ when
$u\leq\delta$, and $B_{s,p}\leq\exp(-\kappa_1)$ when $u\geq\delta$,
where $\kappa_1=-\log\{(1+e^{-2\delta})/2\}>0$.

For a nonconstant vector $\bm v\in\mathbb Z_q^n$ with largest fiber of size $n-s$,
let $\bm R$ be a uniform binary row of weight $d$. Realize $\bm R$ by
selecting $d$ disjoint random pairs of coordinates and independently choosing
one coordinate from each pair. If any pair has unequal labels, then,
conditionally on all pairs and all remaining choices, any given residue of
$\bm v\cdot\bm R$ has probability at most $1/2$.
For $u=sp\leq\delta$, necessarily $p\leq\delta$. The event that exactly
one of the $2d$ paired coordinates lies outside the largest fiber has probability
\[
  \frac{2ds}{n}\prod_{j=0}^{2d-2}\frac{n-s-j}{n-1-j}
 \geq 2u\left\{1-\frac{(2d-1)(s-1)}{n-2d+1}\right\}
 \geq2u(1-4\delta).
\]
On this event some pair has unequal labels. Therefore, for every residue $b$,
\[
 \P[\bm v\cdot\bm R\equiv b\pmod q]
 \leq\exp\{-(1-4\delta)sp\}\qquad(sp\leq\delta).
\]
For $sp\geq\delta$, the same residue probability is at most
$\exp(-\kappa_2)$ for a constant $\kappa_2=\kappa_2(\delta)>0$, uniformly
in $p,n,s,q$. Here is an explicit justification of this uniformity.
If $s\geq n/3$, choose a union of fibers with size between $n/3$ and $2n/3$;
the first pair crosses this union with probability bounded below by $2/9$.
If $s<n/3$, let $W$ count pairs crossing the largest fiber. Then
\begin{align*}
\mathbb EW
&=\frac{2ds(n-s)}{n(n-1)},\\
\mathbb E[W(W-1)]
&=\frac{4d(d-1)s(s-1)(n-s)(n-s-1)}
        {n(n-1)(n-2)(n-3)}.
\end{align*}
For sufficiently large $n$, these give $\mathbb EW\geq4sp/3$ and
$\mathbb EW^2\leq3sp+16(sp)^2$. Thus
\[
 \P[W>0]\geq\frac{(\mathbb EW)^2}{\mathbb EW^2}
 \geq\frac{16\delta}{9(3+16\delta)}>0.
\]
These bounds give the asserted $\kappa_2$ and are independent of the modulus.

We next bound small left-kernel vectors. A nonzero rational left-kernel vector
of support less than $t$ can be scaled to a primitive integer vector and reduced
modulo $2$, producing a nonzero binary left-kernel vector of support less than $t$.
For a prescribed support of size $s$, compare the corresponding $s$ rows with
independent Bernoulli($p$) rows, conditioned on each having weight $d$.
Uniform Stirling bounds give $\P[\operatorname{Bin}(n,p)=d]\geq c_0/\sqrt d$
for an absolute constant $c_0>0$, hence the probability that all column sums
of those $s$ rows are even is at most
\[
 \exp(C_0s\log d) B_{s,p}^{n}
\]
for an absolute $C_0$. Since $d\geq(1+\varepsilon)\log n$,
$\sup_d(\log d)/d\to0$. In the range $sp\leq\delta$, this conditioning
cost is uniformly absorbed into the exponent; for sufficiently large $n$,
\[
 \binom ns\exp(C_0s\log d)B_{s,p}^{n}\leq n^{-\varepsilon s/2}.
\]
For $sp>\delta$ and $s<t$, choosing $a$ so that $C_0a<\kappa_1/4$
gives an exponential bound. The number of supports of size at most $t$ is
$\exp\{o(n)\}$ uniformly in $d$, since $t/n=a/\log d\to0$ uniformly.
Consequently
\[
 \P[\exists\,\bm x\ne0:\ |\operatorname{supp}\bm x|<t,
                  \bm x^\top\bm Q=0]
 \leq C_\varepsilon n^{-\varepsilon/2}+e^{-c_\varepsilon n}.
\]

For right-kernel vectors with a fiber larger than $n-t$, use $q=d-1\geq2$,
as in the proof of Lemma~4.4 in \citet{FKS22}. If a rational kernel vector has this property,
scale it to a primitive integer vector and reduce modulo $q$. The result is
nonzero. It cannot be constant, because $d\equiv1\pmod q$ makes a constant
kernel vector zero. It retains a fiber of size at least that of the original
vector. Any real solution with prescribed equal coordinates can first be
replaced by a nonzero rational solution, because the relevant homogeneous
system has rational coefficients.
For largest fiber $n-s$, the number of residue vectors is at most
$\binom ns q^{s+1}$. Our scalar estimates, applied to the first $n-1$ rows,
give a union bound over $1\leq s<t$.
For $sp\leq\delta$, the cost $(s+1)\log q\leq2s\log d$ and the factor
$1-1/n$ in the number of rows are uniformly absorbed into the exponent,
yielding, for sufficiently large $n$,
\[
 \binom ns q^{s+1}
       e^{-(1-4\delta)sp(n-1)}\leq n^{-\varepsilon s/2}.
\]
For $sp>\delta$, the product probability is at most $e^{-\kappa_2(n-1)}$;
the labeling cost is at most $e^{an+\log d}$, and the support-count cost is
$e^{o(n)}$. Choose $a<\kappa_2/4$. Therefore
\[
 \P[\exists\,\bm x\ne0:\ \bm x\text{ has a fiber larger than }n-t,
                  \bm x\cdot\bm R_i=0\ (1\leq i\leq n-1)]
 \leq C_\varepsilon n^{-\varepsilon/2}+e^{-c_\varepsilon n}.
\]

Let $A_L$ and $A_R$ denote the exceptional probabilities in the preceding
left- and right-kernel estimates, respectively. Lemma~2.1 in \citet{FKS22} holds for any field, any random matrix with independent and identically distributed rows, any property of vectors, and any $t$; we apply it over the rational numbers with
the property that every fiber has size at most $n-t$. Their Lemma~4.2 states that, for $0\leq d\leq n/2$ and a vector whose largest fiber has size $n-s$, the probability that its scalar product with a uniform binary vector of weight $d$ equals any given value is $O(\sqrt{n/(sd)})$. The implied constant is absolute, since their proof uses only a Chernoff bound and the Erd\H{o}s--Littlewood--Offord theorem. Hence,
uniformly for vectors with that property, the scalar atom probability is
at most $C\sqrt{n/(td)}$ for an absolute constant $C$. Thus
\[
 \P[\bm Q\text{ is singular}]
 \leq A_L+\frac nt A_R+\frac nt\,C\sqrt{\frac{n}{td}}.
\]
Substituting $t=an/\log d$, the third term is
\[
 \frac nt\,C\sqrt{\frac{n}{td}}
 =C_\varepsilon\frac{(\log d)^{3/2}}{\sqrt d}.
\]
Combining these estimates proves the stated bound, with the weaker exponent
$\gamma_\varepsilon=\varepsilon/4$ absorbing all exponentially small terms.
\end{proof}

%% file: rank-growing.tex
The proof reduces the full column rank of the $m\times n$ design matrix $\bm X$ to the nonsingularity of an $n\times n$ matrix $\bm Q$ to which Lemma~\ref{lm:quantitative-rank} applies. If the group size $k$ is at least $(1+\varepsilon)\log n$, for a small fixed $\varepsilon>0$, then $\bm Q$ consists of the first $n$ rows of $\bm X$. Otherwise, each row of $\bm Q$ is the sum of a block of rows of $\bm X$ whose groups are pairwise disjoint. The converse uses the number of null columns.

\paragraph{Step 1: the reduction}
Suppose that $\bm Q=\bm B\bm X$ for an $n\times m$ matrix $\bm B$, that is, every row of $\bm Q$ is a linear combination of rows of $\bm X$. Since the rank of a product is at most the rank of each factor, $\rank(\bm Q)\leq\rank(\bm X)$. Hence, if $\bm Q$ is nonsingular, then $\rank(\bm X)\geq n$, and $\bm X$ has full column rank. It therefore suffices to construct such a matrix $\bm Q$ that is nonsingular with the stated probability.

Choose fixed constants $\varepsilon\in(0,1)$ and $\zeta>0$ so that $2(1+\varepsilon)(1+\zeta)<c$, and put $L=(1+\varepsilon)\log n$. We distinguish two cases.

\paragraph{Step 2: group size $k\geq L$}
Let $\bm Q$ consist of the first $n$ rows of $\bm X$, which is possible because $m\geq n$; this corresponds to $\bm B=(\bm I\ \bm0)$. The rows of $\bm Q$ are independent, and each is the incidence vector of a uniformly random $k$-subset of $V$. So $\bm Q$ is the matrix of Lemma~\ref{lm:quantitative-rank} with weight $d=k$. Since $k\geq L$, and $k\leq n/2$ for sufficiently large $n$ by the assumption $k=o(n/\log n)$, the lemma gives
\[
\P[\bm Q\text{ is singular}]\leq C_\varepsilon\left(n^{-\gamma_\varepsilon}\log k+\frac{(\log k)^{3/2}}{\sqrt k}\right).
\]
This case uses only $m\geq n$. For the final statement of the theorem, choose $\varepsilon\leq\xi$, so that $k\geq(1+\xi)\log n$ implies $k\geq L$.

\paragraph{Step 3: group size $k<L$, and the blocks}
The rows of $\bm X$ now have weight $k<L$, which is too small for Lemma~\ref{lm:quantitative-rank}. We build rows of larger weight by adding rows of $\bm X$ with disjoint groups. Set $\nu =\lceil L/k\rceil$ and $d=\nu k$, so that
\[
L\leq d<L+k<2L.
\]
Take $G=\lceil(1+\zeta)n\rceil$ disjoint blocks of $\nu $ consecutive rows of $\bm X$. There are enough rows for this: since $\nu <L/k+1<2L/k$, for sufficiently large $n$,
\[
G\nu\leq\bigl((1+\zeta)n+1\bigr)\frac{2L}{k}<c\,\frac nk\log n\leq m.
\]
Call a block successful when its $\nu $ groups are pairwise disjoint. The sum of the $\nu $ rows of a successful block is then a binary vector with exactly $d$ entries equal to $1$, namely the incidence vector of the union of the $\nu $ groups.

\paragraph{Step 4: most blocks are successful}
For two independent uniformly random $k$-sets, the expected size of their intersection is $k^2/n$, so the probability that they intersect is at most $k^2/n$. A union bound over the $\binom\nu2$ pairs of groups in a block gives
\[
\P[\text{a block is not successful}]\leq\binom\nu2\frac{k^2}{n}\leq\frac{d^2}{2n}=O\left(\frac{(\log n)^2}{n}\right).
\]
The expected number of unsuccessful blocks among the $G$ blocks is therefore $O((\log n)^2)$. At least $n$ blocks are successful unless more than $G-n\geq\zeta n$ blocks fail, so by Markov's inequality, the probability that fewer than $n$ blocks are successful is $O((\log n)^2/n)$.

\paragraph{Step 5: the matrix $\bm Q$}
On the event that at least $n$ blocks are successful, let $\bm Q$ be the $n\times n$ matrix whose rows are the sums of the first $n$ successful blocks. Each row of $\bm Q$ is a sum of rows of $\bm X$, so $\bm Q=\bm B\bm X$ for a $0$--$1$ matrix $\bm B$, and Step~1 applies. We show that, conditional on which blocks are successful, $\bm Q$ is the matrix of Lemma~\ref{lm:quantitative-rank} with weight $d$. The blocks consist of disjoint sets of rows of $\bm X$, and the rows of $\bm X$ are independent, so the blocks are independent, and conditioning on the success indicators keeps the successful blocks independent. For a single block, conditional on success, the union of its groups is a uniformly random $d$-subset of $V$: its distribution is invariant under every permutation of the vertices, and the permutations act transitively on the $d$-subsets. Hence the rows of $\bm Q$ are, conditionally, independent incidence vectors of uniformly random $d$-subsets. Since $d\geq L$ and $d<2L$, we also have $d\leq n/2$ for sufficiently large $n$, and Lemma~\ref{lm:quantitative-rank} gives
\[
\P[\bm Q\text{ is singular}\mid\text{success indicators}]\leq C_\varepsilon\left(n^{-\gamma_\varepsilon}\log d+\frac{(\log d)^{3/2}}{\sqrt d}\right)
\]
whenever at least $n$ blocks are successful.

\paragraph{Step 6: the rate}
In both cases, by Step~1,
\begin{align*}
\P[\rank_{\reals}(\bm X)<n]
&\leq\P[\text{fewer than $n$ blocks are successful}]\\
&\quad+\P[\bm Q\text{ is singular and at least $n$ blocks are successful}],
\end{align*}
where the first term is absent in Step~2. The function $(\log d)^{3/2}/\sqrt d$ decreases for $d>e^3$, and $d\geq L$ in both cases, so $(\log d)^{3/2}/\sqrt d\leq(\log L)^{3/2}/\sqrt L$; moreover, $\log d\leq\log n$. By Steps~2, 4 and~5, the probability that $\bm X$ is rank deficient is therefore at most
\[
C_c\left(\frac{(\log n)^2}{n}+n^{-\gamma_\varepsilon}\log n+\frac{(\log\log n)^{3/2}}{\sqrt{\log n}}\right),
\]
for a constant $C_c$ depending only on $c$, through $\varepsilon$ and $\zeta$. The last term dominates, proving the asserted rate.

\paragraph{Step 7: the converse}
A design matrix with fewer than $n$ rows cannot have rank $n$, so $m\geq n$ is necessary. For the second necessary condition, let $Z$ be the number of zero columns of $\bm X$, that is, of vertices that take part in no experiment; a zero column precludes full column rank. If $m\leq c'(n/k)\log n$ with $c'<1$, then
\[
\E Z=n(1-k/n)^m\geq n^{1-c'}(1+o(1))\longrightarrow\infty,
\]
where $k\log n/n\to0$ ensures that the quadratic remainder in $\log(1-k/n)$ is $o(1)$. For two distinct columns, the probability that both are zero is
\[
\left(\frac{(n-k)(n-k-1)}{n(n-1)}\right)^m\leq(1-k/n)^{2m},
\]
so the zero-column indicators have nonpositive pairwise covariances, and $\operatorname{Var}Z\leq\E Z$. By Chebyshev's inequality, $\P[Z=0]\leq\operatorname{Var}Z/(\E Z)^2\leq1/\E Z\to0$, so $\bm X$ fails to have full column rank with probability tending to one.

%% file: rank-fixed.tex
The theorem has two parts: an upper bound on the probability that $\bm X$ is rank deficient when $c>1$, and a matching lower bound, which also covers $c<1$.

For the upper bound, recall that $\bm X$ has full column rank if and only if the only vector $\bm\alpha$ with $\bm X\bm\alpha=\bm0$ is $\bm\alpha=\bm0$. Rank deficiency is therefore the event that $\bm X$ has a nonzero kernel vector, and we bound the probability of this event. Over the real numbers, a union bound over candidate kernel vectors is impossible, because there are infinitely many of them. We therefore work over a finite field $\mathbb F_p$. Rank deficiency over the real numbers implies rank deficiency over $\mathbb F_p$, because rank is detected by the $n\times n$ minors of $\bm X$, which are integers, and a minor that is zero over the real numbers is also zero modulo $p$; we give the details below. It therefore suffices to bound the probability that $\bm X$ has a nonzero kernel vector in $\mathbb F_p^n$, a set of at most $p^n$ candidates. We group the candidates by the size $s$ of their support. A candidate with small support is excluded by any row of $\bm X$ that meets its support in exactly one vertex. For $s=1$, the absence of such a row means that $\bm X$ has a null column, and this case gives the leading term $n^{-(c-1)}$. A candidate with large support is excluded by any row whose vertices all carry the same frequent nonzero label.

For the lower bound, a null column makes $\bm X$ rank deficient, and a second-moment argument for the number of null columns shows that this event has probability at least $(1+o(1))n^{-(c-1)}$ when $c>1$, and probability tending to one when $c<1$.

Choose a prime $p>k$ and view $\bm X$ over the field $\mathbb F_p$ of integers modulo $p$; $p$ is prime so that $\mathbb F_p$ is a field.
Over any field, a matrix with $n$ columns and at least $n$ rows has rank $n$ if and only if at least one of its $n\times n$ minors, the determinants of the submatrices formed by $n$ of its rows, is nonzero. A determinant is a polynomial in the entries of the matrix with integer coefficients. Since the entries of $\bm X$ are $0$ and $1$, each $n\times n$ minor of $\bm X$ is an integer $D$, and the same minor computed over $\mathbb F_p$ equals $D$ modulo $p$. If $\bm X$ is rank deficient over the real numbers, every such $D$ is zero, hence zero modulo $p$, and $\bm X$ is rank deficient over $\mathbb F_p$. Equivalently, full rank over $\mathbb F_p$ implies full rank over the real numbers. The converse may fail, since a nonzero minor can be divisible by $p$. Thus it suffices for the upper bound
to exclude nonzero vectors $\bm\alpha\in\mathbb F_p^n$ with $\bm X\bm\alpha=0$.

For a support $S\subseteq[n]$ with $|S|=s$, a row meeting $S$ in exactly one
coordinate has nonzero scalar product with every vector whose support is $S$.
The probability of such a row is
\[
 a_s=\frac{s\binom{n-s}{k-1}}{\binom nk}
     =\frac{ks}{n}\prod_{i=0}^{k-2}\frac{n-s-i}{n-1-i}.
\]
Consequently, the probability that any kernel vector has support $S$ is at most
$(1-a_s)^m\leq\exp(-ma_s)$. This bound is independent of the labels on $S$,
so the union bound involves $\binom ns$ supports, rather than all their labellings.

Put $h_n=\lfloor n/(\log n)^2\rfloor$. Uniformly for $1\leq s\leq h_n$,
\[
 a_s=\frac{ks}{n}\{1+O(s/n+1/n)\},\qquad
 ma_s=cs\log n+O(s/\log n).
\]
It follows that
\[
 \binom ns e^{-ma_s}
 \leq \left(\frac{e\,n^{-(c-1)}e^{O(1/\log n)}}s\right)^s.
\]
For $c>1$, summing this bound over $2\leq s\leq h_n$ gives
$O(n^{-2(c-1)})$. For $s=1$, the exact bound is
$n(1-k/n)^m$.

Choose a fixed $\delta>0$ small enough that $(1-\delta)^{k-1}>1/2$.
For $h_n<s\leq\delta n$ and sufficiently large $n$ we have
$a_s\geq ks/(2n)$ and
\[
 \log\binom ns\leq s\log(en/s)\leq s(2\log\log n+O(1)).
\]
Thus the sum of the support bounds in this range is
$\exp\{-\Omega(n/\log n)\}$.

Finally, if $s>\delta n$, some nonzero label $a\in\mathbb F_p$ occurs at least
$t\geq\delta n/(p-1)$ times. A row selecting $k$ of these coordinates has scalar
product $ka\ne0$, because $a\neq0$, $k\neq0$ in $\mathbb F_p$ as $1\leq k<p$, and $\mathbb F_p$ has no zero divisors. For sufficiently large $n$, its probability is at least
\[
 \frac{\binom tk}{\binom nk}
 \geq\left(\frac{t-k+1}{n}\right)^k
 \geq\rho:=\left(\frac{\delta}{2(p-1)}\right)^k>0.
\]
There are at most $p^n$ labeled vectors in this last range, so their total
probability of being in the kernel is at most
\[
 p^n(1-\rho)^m\leq\exp\{n\log p-\rho m\}
                      =\exp\{-\Omega(n\log n)\}.
\]
We have proved, for $c>1$,
\[
 \P[\operatorname{rank}_{\mathbb R}(\bm X)<n]
 \leq n(1-k/n)^m+O(n^{-2(c-1)})+\exp\{-\Omega(n/\log n)\}
      =(1+o(1))n^{-(c-1)}.
\]

For the matching lower bound, let $Z$ be the number of zero columns.
Then
\begin{align*}
\E Z&=n(1-k/n)^m=(1+o(1))n^{-(c-1)},\\
\E[Z(Z-1)]&=n(n-1)\left(\frac{(n-k)(n-k-1)}{n(n-1)}\right)^m\\
&=(1+o(1))n^{-2(c-1)}.
\end{align*}
For $c>1$, Bonferroni gives
$\P[Z>0]\geq\mathbb EZ-\tfrac12\mathbb E[Z(Z-1)]
=(1+o(1))n^{-(c-1)}$; a zero column forces real rank deficiency.
For $c<1$, note that
\[
 \frac{(n-k)(n-k-1)}{n(n-1)}
 =\left(1-\frac kn\right)^2
   \left(1-\frac{k}{(n-1)(n-k)}\right)
 \leq\left(1-\frac kn\right)^2.
\]
The zero-column indicators therefore have nonpositive pairwise covariances,
so $\operatorname{Var}Z\leq\mathbb EZ$. Chebyshev's inequality gives
$\P[Z=0]\leq1/\mathbb EZ\to0$. The final claim follows by monotonicity in $m$.